\documentclass[12pt, reqno]{amsart}
\usepackage{}
\usepackage{amsmath, amstext, amsbsy, amssymb, amscd}
\usepackage{amsmath}

\usepackage{bbm}
\usepackage{dsfont}
\usepackage{yfonts}
\usepackage{amsmath}
\usepackage{amssymb}
\usepackage{amsfonts}
\usepackage{mathrsfs}
\usepackage{amsxtra}
\usepackage{amscd}
\usepackage{amsthm}
\usepackage{amsfonts}
\usepackage{amssymb}
\usepackage{eucal}
\usepackage{color}
\usepackage{mathrsfs}
 \usepackage[all]{xy}
\usepackage[CJKbookmarks=true]{hyperref}

\newtheorem{theorem}{Theorem}[section]
\newtheorem{lemma}[theorem]{Lemma}
\newtheorem{prop}[theorem]{Proposition}
\newtheorem{corollary}[theorem]{Corollary}
\theoremstyle{definition}

\newtheorem{defn}[theorem]{Definition}
\newtheorem{example}[theorem]{Example}
\newtheorem{remark}[theorem]{Remark}

\numberwithin{equation}{section}

\def\bz{{\bar 0}}
\def\bo{{\bar 1}}

\def\codim{\text{codim}}
\def\ad{\text{ad}}

\def\GL{{\text{GL}}}

\def\Spec{\text{Spec}}

\def\gr{\mathrm{gr}}
\def\F{\text{F}}

\def\Prim{\mathrm{Prim}}

\def\Ann{\mathrm {Ann}}

\def\bbz{\mathbb Z}
\def\bbc{\mathbb C}

 \def\LIM{ \varinjlim }

\def\cu{\mathcal {U}}

\def\F{\mathrm {F}}

\def\ad{\mathrm {ad}}

\def\spec{\mathrm {Spec}}
\def\deg{\mathrm {deg}}

\def\fin{\mathrm {cofin}}

\def\abz{A_{\bar{0}}}
\def\abo{A_{\bar{1}}}

\def\gbz{\mathfrak{g}_{\bar{0}}}
\def\gbo{\mathfrak{g}_{\bar{1}}}

\def\mmm{{\mathfrak m}}
\def\ppp{{\mathfrak p}}
\def\ggg{{\mathfrak g}}

\def\vvv{{\mathfrak v}}

\def\bz{{\bar 0}}
\def\bo{{\bar 1}}

\makeatletter
\newcommand{\rmnum}[1]{\romannumeral #1}
\newcommand{\Rmnum}[1]{\expandafter\@slowromancap\romannumeral #1@}
\makeatother

{\vskip-\lastskip\medskip
  \noindent
  {\em #1.}\enspace
  }%
{\qed\par\medskip
  }

\begin{document}

\title[Super formal Darboux-Weinstein theorem and finite W-superalgebras]{Super formal Daboux-Weinstein theorem and finite $W$-superalgebras}

\author{Bin Shu, Husileng Xiao}
\address{School of mathematical Science, Harbin Engineering University, Harbin, 15001, China.}\email{hsl1523@163.com; hslxiao@hrbeu.edu.cn}
\address{School of mathematical Science, East China Normal University, Shanghai, 200241, China.}\email{bshu@math.ecnu.edu.cn}
\thanks{   }
\maketitle
\begin{abstract}
Let $\vvv=\vvv_{\bar{0}}+\vvv_{\bar{1}}$ be a $\mathbb{Z}_2$-graded (super) vector space with an even $\mathbb{C}^{\times}$-action and $\chi \in \vvv_{\bar{0}}^{*}$ be a fixed point of the induced action.
In this paper we  prove an equivariant Darboux-Weinstein theorem for the formal polynomial algebras $\hat{A}=S[\vvv_{\bar{0}}]^{\wedge_{\chi}}\otimes \bigwedge(\vvv_{\bar{1}})$. We also give a quantum version of it. Let $\ggg=\ggg_{\bar{0}}+\ggg_{\bar{1}}$ be a basic Lie superalgebra and $e \in \ggg_{\bar{0}}$ be a nilpotent element. We use the equivariant quantum Darboux-Weinstein theorem to give a Poisson geometric realization of the finite $W$-superalgebra $\mathcal{U}(\ggg,e)$ in the sense of Losev. An indirect relation between finite  $W$-(super)algebras $\mathcal{U}(\ggg,e)$ and $\mathcal{U}(\ggg_{\bar{0}},e)$ is presented. Finally we use such a realization to study the finite-dimensional irreducible modules over $\mathcal{U}(\ggg,e)$.
\end{abstract}

\section{introduction}

\subsection{ Darboux-Weinstein theorems}

In \cite{Wei}, the author proves the Darboux-Weinstein (DW) theorem, which states that each Poisson manifold is locally the product of a symplectic manifold and a Poisson manifold having a point where the rank is zero. We recall the formal version of this theorem.

Let $\hat{A}_{\bar{0}}=S[[\vvv_{\bar{0}}]]$ be the formal power series algebra of a vector space $\vvv_{\bar{0}}$ over $\mathbb{C}$  and $\omega:=\{\cdot, \cdot \}$ be a continuous Poisson bracket on $\hat{A}_{\bar{0}}$. The DW theorem says that $\hat{A}_{\bar{0}}$ can be split as a product
$$\hat{A}_{\bar{0}}\cong S[[V_{\bar{0}}]]\otimes \hat{B}_{\bar{0}}$$
of Poisson algebras (see \cite{Kal} for the proof). Here the Poisson product on the formal power algebra $S[[V_{\bar{0}}]]$ arises naturally from a symplectic subspace $V_{\bar{0}}\subset \vvv_\bz$ on which the restriction of the Poisson bi-vector is non-degenerate,  while  $\hat{B}_\bz$ is the centralizer of $V_\bz$  in $\hat{A}_{\bar{0}}$ with respect to the Poisson product (c.f. Example \ref{example1}).

This theorem plays a fundamental role in the theory of symplectic singularities (c.f. \cite{Kal}). In the recent works of Losev on the representations of certain quantum algebras, the quantum DW (qDW) theorem is always a powerful tool. Those quantum algebras include finite W-algebras, symplectic reflection algebras and rational Cherendick algebras, and arise from the formal quantization of Poisson algebras.

\subsubsection{} In the present paper we consider an equivariant super DW (esDW) theorem  and its quantum version. Then we use it to study the representations of finite $W$-superalgebras.

\begin{defn}
Let  $A=\abz+\abo$ be a $\mathbb{Z}_{2}$-graded commutative algebra over the complex number field $\mathbb{C}$.
A Poisson bracket $\{\cdot, \cdot \}$ on $A$ is a linear map $A \otimes A \rightarrow A$ with
\begin{itemize}
\item[(1)] $\{\cdot, \cdot \}$ is a Lie superalgebra;
\item[(2)]  $\{f,gh\}=\{f,g\}h+(-1)^{|f||g|}g\{f,h\}$ for all homogenous $f,g,h \in A$.
\end{itemize}
A Poisson bracket $\{\cdot, \cdot \}$ is called even (\textit{resp}. odd) if it is an even (\textit{resp}. odd) linear map between vector superspaces.
\end{defn}

\begin{example} \label{example1}
\begin{itemize}
\item[(1)]  Let $\vvv$ be a vector superspace of dimension $(2n,m)$ equipped with a super symplectic form $\omega$.
Then there is a standard Poisson bracket on $S[\vvv]$  given by $\{x , y\}= \omega(x,y)$ for all $x,y \in \vvv$.
\item[(2)] Let $\ggg$ be a Lie superalgebra, then $S[\ggg]$ also has a standard Poisson structure given by
 $\{x , y \}= [x,y]$ for all $x,y \in \ggg$.
\end{itemize}
\end{example}

Let $\vvv=\vvv_{\bar{0}}+\vvv_{\bar{1}}$ be a vector superspace and $A$ be the super symmetric polynomial algebra $S[\vvv]=S[\vvv_{\bar{0}}]\otimes \bigwedge(\vvv_{\bar{1}})$. For a fixed $\chi \in \vvv_{\bar{0}}^*$, let $M_{\chi}$ be the maximal ideal of $S[\vvv]$ corresponding to $\chi$ i.e.
$M_{\chi}$ is the ideal generated by $x-\chi(x)$ for $x \in \vvv_{\bar{0}} \oplus \vvv_{\bar{1}}$.
Let $A^{\wedge_\chi}$ be the completion of $A$ at $\chi$. Namely $A^{\wedge_\chi}=\varprojlim A/AM_\chi^{k}$ is the topological algebra with the ideals  $AM^{k}_\chi$ forming a set of fundamental neighborhoods of $0 \in A^{\wedge_\chi}$. From now on, the formal power series superalgebra  $A^{\wedge_\chi}$ for a fixed $\chi$ is denoted by $\hat{A}$ in the present paper.

Suppose that there is a continuous $\mathbb{C}^{\times}$-action on $\hat{A}$ which arises from an even linear algebraic $\mathbb{C}^{\times}$-action on $\vvv$. Let $\chi \in \vvv_{\bar{0}}^*$ be a fixed point for the induced action.
Assume that there is an integer $ k \in \mathbb{Z}$ such that $\{t\cdot f , t\cdot g \}=t^{k}\{f , g\}$ for all $f,g \in \hat{A}$ and $t \in \mathbb{C}^{\times}$. Let $V \subset \mathcal{T}_\chi^{*}(\spec(A)) = \vvv$ be a subspace such that the restriction of the Poisson bi-vector $\Pi$ (see \eqref{bivectr eq} for the definition) on $V$ is non-degenerate. In  \S 2 we prove the following $\mathbb{C}^{\times}$-equivariant  DW theorem for the formal  polynomial superalgebra $\hat{A}$ with an even Poisson bracket.

\begin{theorem}[Equivariant super DW (esDW) theorem]\label{EDWthm}
In the above setting, we have the following $\mathbb{C}^{\times}$-equivariant isomorphism of Poisson superalgebras
$$ \phi: S[[V]]\otimes  \hat{B} \longrightarrow \hat{A},$$
where $\hat{B}$ is the centralizer of $V$ in $\hat{A}$ with respect to $\{\cdot , \cdot \}$, which is a Poisson subalgebra of $\hat{A}$, and the restriction of $\phi$ on $\hat{B}$ is the identity map.
\end{theorem}

We prove the theorem by constructing such an isomorphism explicitly. In the case if $V= \vvv$ and $\vvv$ is pure even, our proof is more direct than the one in \cite{Kal}. In the super setting, there is a significant difference between  equivariant and non-equivariant cases.

\subsubsection{}\label{sec 1.1} For an associative superalgebra $A$, denote by $A[[\hbar]]$ the vector space of formal power series of variable $\hbar$ with coefficients in $A$.
\begin{defn}
Let $A$ be Poisson superalgebra, a \textit{formal  quantization} of $A$ is a pair $(A[[\hbar]], \ast)$, where
the  star product $\ast : A[[\hbar]] \otimes_{\mathbb{C}[[\hbar]]} A[[\hbar]] \longrightarrow A[[\hbar]]$ satisfies
\begin{itemize}
\item[(1)] $\ast$ is associative;
\item[(2)] for all $\mathbb{Z}_2$-homogenous $f,g \in  A $,
$$ f\ast g-fg \in A[[\hbar]] \hbar^{2}\quad\footnote{Usually $\hbar^{2}$ is replaced by $\hbar$, here we need such a slight modification for its apllication. }\quad  \text{and} \quad f\ast g -(-1)^{|f||g|}g\ast f-\{f,g\}\hbar^{2} \in A[[\hbar]] \hbar^3
.$$
\end{itemize}
\end{defn}
For a Poisson superalgebra $A$ with a $\mathbb{C}^{\times}$-action, by a \textit{ $\mathbb{C}^{\times}$-equivariant quantization } $(A[[\hbar]], \ast)$ we mean that there is a $\mathbb{C}^{\times}$-action on $A[[\hbar]]$ such that the quotient map  $A[[\hbar]] \rightarrow A[[\hbar]]/A[[\hbar]]\hbar=A$ is $\mathbb{C}^{\times}$-equivariant.

\begin{example}\label{ex quan}
\begin{itemize}
\item[(1)] Let $(\vvv,\omega)$ be a symplectic  superspace. Denote by $(\mathbf{A}_{\hbar}(\vvv),\ast)$  the
$\mathbb{C}[[\hbar]]$-algebra generated by basis $v_1,v_2, \ldots, v_{n}$ of $\vvv$ with commutating relations
$$v_i \ast v_j-(-1)^{|v_i||v_j|}v_j \ast v_i=\omega(v_i,v_j)\hbar^{2} $$
for all $ 1 \leq i ,j\leq n$. Taking $\hbar=1$, we obtain the Weyl algebra $\mathbf{A}(\vvv)$ from $S[\mathfrak{v}][\hbar] \subset \mathbf{A}_{\hbar}(\vvv)$. We have the following isomorphism of $\mathbb{C}[[\hbar]]$ spaces
$$S[\vvv][[\hbar]] \longrightarrow \mathbf{A}_{\hbar}(\vvv) \quad  v_{1}^{i_{1}} \cdots v_{n}^{i_n} \mapsto v_{1}^{i_{1}} \ast \cdots \ast v_{n}^{i_n}.$$
Under this identification, it is easy to check that $\mathbf{A}_{\hbar}(\vvv)$ is a quantization of the Poisson algebra $S[\vvv]$.

\item[(2)] For a Lie superalgebra $\ggg$,  the Poisson algebra $S[\ggg]$ has a standard quantization $(S[\ggg][[\hbar]],\ast)$, which is the quotient of the formal  tensor  algebra $T(\ggg)[[\hbar]]$ modulo the two sided ideal generated by
$$x\otimes y -(-1)^{|x||y|}y\otimes x =[x,y]\hbar^2 $$
for all the $x,y \in \ggg$.
Taking $\hbar=1$ we can recover the  universal enveloping algebra  $\mathcal{U}(\ggg)$ from $S[\ggg][\hbar] \subset S[\ggg][[\hbar]]$.
\end{itemize}
\end{example}

\begin{theorem}[Super quantum equivariant DW (sqeDW) theorem] \label{SQDWtm}
Let  $\hat{A},V $ be as in Theorem \ref{EDWthm} and  $(\hat{A}[[\hbar]],\ast )$ be a $\mathbb{C}^{\times}$-equivariant quantization of $\hat{A}$.
Then there exists a $\mathbb{C}^{\times}$-equivariant isomorphism of $\mathbb{C}[[\hbar]]$ algebras
$$ \Phi_{\hbar}:  \mathbf{A}_{\hbar}(V)^{\wedge_0} \otimes_{\mathbb{C}[[\hbar]]} \hat{B}_{\hbar} \longrightarrow \hat{A}[[\hbar]] $$
Here $\hat{B}_{\hbar}$ is a quantization of $\hat{B}$.
\end{theorem}

We  identify $S[[V]]$ with a Poisson subalgebra of $\hat{A}$ in the proof of Theorem \ref{EDWthm}. Under this identification  the isomorphism $\phi$ is
given by $\phi(a\otimes b)=ab$ for all $a,b \in \hat{A}$. The isomorphism $\Phi_{\hbar}$ is constructed in a similar way.

\subsection{ Finite $W$-superalgebra via quantum DW theorem}\label{sec introduction}

 For a complex reductive Lie algebra $\ggg$ and a nilpotent $e$ in it, one has the finite $W$-algebra $\mathcal{U}(\ggg, e)$, see \cite{Pr1} for the definition. There are several different but equivalent  definitions for the finite $W$-algebras. In \cite{Lo1}, the author constructed the finite $W$-algebra from the Fedosov quantization. Via this construction, Losev relates the representation theoretic objects of the $\mathcal{U}(\ggg)$ and $\mathcal{U}(\ggg,e)$.  Those objects include the prime (primitive) ideals, Harish-Chandra bimodules etc.   Using those correspondences, Losev  classifies  the finite-dimensional modules for classical $W$-algebra $\mathcal{U}(\ggg,e)$ (c.f. \cite{Lo1}).

 \subsubsection{} In \cite{WZ}, \cite{ZS1} the authors study the finite $W$-superalgebras. Let us recall the definition of finite $W$-superalgebras. Let $\mathfrak{g}=\gbz+\gbo$ be a basic Lie superalgebra with an \textit{even} non-degenerate super-symmetric invariant bilinear form (say Killing form ) $(\cdot,\cdot)$. Denote by $\mathcal{U}$ the universal enveloping algebra of $\ggg$.
Let $e \in \gbz$ be a nilpotent element and  $\chi \in \mathfrak{g}^*$ be the unique element with $\chi(X)=(e, X) \text{ for all } X \in \mathfrak{g}$. Since the Killing form is even, we have $\chi(x)=0$ for all $x \in \ggg_{\bar{1}}$.

Let $\ggg=\bigoplus_{i}\ggg(i)$ be a good $\mathbb{Z}$-grading for $e$, see \cite{Hoy} for the definition.
It follows from the property of Killing form that the super symplectic form $\langle \cdot,\cdot\rangle$ on $\ggg(-1)$ given by
$$\langle x, y\rangle=\chi([x,y]) \text{ for all $x,y \in \ggg(-1)$ }$$
is non-degenerate.

Let
$\ppp=\bigoplus_{i \geq 0}\ggg(i)$, choose a Lagrangian subspace (isotropic subspace of maximal dimension) $l$ of $\ggg(-1)$ and set $\mmm=\bigoplus_{i \leq -2}\ggg(i) \oplus l$. Let $\ggg' $ (\textit{resp}. $\mmm'$) be the subspace of $\mathcal{U} $  consisting of  $\{ x-\chi(x) \mid x \in \ggg \text{ (\textit{resp}. $\mmm$)} \}$. Since $\chi(x)=0$ for all $x \in \ggg_{\bar{1}}$, $\ggg'$ (\textit{resp}. $\mmm'$) is a $\mathbb{Z}_2$-graded vector space.

Denote by $I_{\chi}$ the left ideal of $\mathcal{U}$ generated 
by $\mmm'$.  Then the finite $W$-superalgebra $\mathcal{U}(\ggg,e)$ associate to the nilpotent element $e$ is defined (c.f. \cite[\S2.2.3]{ZS2}) by
\begin{align}\label{def: u(g,e)}
\mathcal{U}(\ggg,e):=(\mathcal{U}/I_{\chi})^{\ad\mmm}=\{ \overline{y} \in \mathcal{U}/I_{\chi} \mid (a-\chi(a))y \in I_{\chi} \text{ \quad for all $a \in \mmm$}\},
\end{align}
which is equal to $(\mathcal{U}/I_{\chi})^{I_\chi}$.
In Section \ref{sec realzation of W} (see Theorem \ref{realzation of W}), we prove that the finite $W$-superalgebra $\mathcal{U}(\ggg,e)$ is  isomorphic to the super quantum Darboux Weinstein slices $\mathscr{W}_{\chi}$ (which will often simply denoted by $\mathscr{W}$). This generalizes the non-super results from \cite{Lo1}, \cite{Los}. We use a similar idea of the proof as in  \textit{loc. cit.}. A difference in our approach is that we  use the explicit construction in the proof of Theorem \ref{SQDWtm} but not the theory of Fedosov quantization and invariant theory. As an application we
reobtain the PBW base theorem of finite W-superalgebras, which is a main result of \cite{ZS1}.

\subsubsection{}\label{sec: completion 0}
Let $\mathcal{U}^{\wedge}_{\mmm'}$ be the completion of $\mathcal{U}$ with respect to the fundamental system $\mathcal{U}\mmm'^{k}$  i.e. $\mathcal{U}^{\wedge}_{\mmm'}=\varprojlim \mathcal{U} /(\mathcal{U}\mmm'^{k})$.
It follows from the definition of the good $\mathbb{Z}$-grading that $\dim([\ggg,f]_{\bar{1}})$ is odd if and only if $\dim(\ggg(-1)_{\bar{1}})$ is
odd. In this case there exists $\Theta \in \ggg(-1)_{\bar{1}} $ with $\langle \Theta , \Theta \rangle=1$.
Let $\bar{\mmm}$ be a maximal isotropic subspace of  symplectic superspace $[\ggg,f]$  subject to $\langle \Theta , \bar{\mmm} \rangle =0$ whenever $\dim([\ggg,f]_{\bar{1}})$ is odd, and $V$ be the subspace $\bar{\mmm}^* \oplus \bar{\mmm}$.

Denote by $\mathbf{A}(V)_{\bar{\mmm}}^{\wedge}\otimes \mathscr{W}$ the completion of  $\mathbf{A}(V)\otimes \mathscr{W}$  with respect to  the fundamental system $\mathbf{A}(V)\otimes \mathscr{W}/(\mathbf{A}(V)\otimes \mathscr{W})\bar{\mmm}^{k}$. Using the realization introduced above we prove the following splitting theorem.

\begin{theorem}\label{realzn of w-alge thm}
There exists an isomorphism
$$ \Phi: \mathbf{A}(V)_{\bar{\mmm}}^{\wedge}\otimes \mathscr{W} \cong  \mathcal{U}^{\wedge}_{\mmm'}$$
of associative algebras.
\end{theorem}

\subsubsection{} As applications of the above splitting theorem,  we firstly give an alternative proof of the super Skryabin's equivalence (c.f. Theorem \ref{Skryabin's equivalence}) which was stated in \cite{ZS1}.

Secondly,  we establish a relation between finite $W$-algebras and $W$-superalgebras. For a Lie superalgebra $\ggg=\ggg_{\bar{0}}\oplus \ggg_{\bar{1}}$, there is a natural  embedding $\mathcal{U}(\ggg_{\bar{0}}) \hookrightarrow \mathcal{U}(\ggg)$ of the universal enveloping algebras. Almost all the results on representations of Lie superalgebras rely on this relation.  It is very useful  to find a similar story for finite $W$-superalgebras. Unfortunately, there is no canonical embedding of $\mathcal{U}(\ggg_{\bar{0}}, e)$ into $\mathcal{U}(\ggg,e)$ for general nilpotent $e$. What is  a good remedy, in
\S\ref{rel of W's} we show that the $W$-algebra $\mathcal{U}(\ggg_{\bar{0}}, e)$ can be  embedded into a larger  associative algebra $\mathcal{A}_{\ddag}=\mathrm{Cl}(V) \otimes \mathcal{U}(\ggg,e)$ (see Theorem \ref{thm rel of W's}). This enables $\mathcal{U}(\ggg_{\bar{0}}, e)$ to play a role in the representation theory of $\mathcal{U}(\ggg,e)$ as $\ggg_{\bar{0}}$ does in the representation theory of $\ggg$ (this will appear in \cite{SXZ}).  For an immediate consequence of this fact, see Corollary \ref{cor 3.12}.
 Finally we  point out that the associative algebra $\mathcal{A}_{\ddag}$ should coincide with  $\mathcal{A}_{\dag}$ in  \cite[\S6.3.2]{Lo4}. Thus Theorem \ref{thm rel of W's} strengthens the claim from \textit{loc. cit.} saying that $\mathcal{A}_{\dag}$ is a Clifford algebra over $\mathcal{U}(\ggg,e)$.

\subsection{Finite-dimensional representations of finite $W$-algebras}
The relation introduced above will give us a powerful tool to investigate representations of  finite $W$-superalgebras. In the present work, we give some information on the set of finite-dimensional irreducible $\mathcal{U}(\ggg,e)$-modules by relating two sided ideals of  $\mathcal{U}(\ggg,e)$ and $\mathcal{U}$. Denote by $\mathfrak{id}(\mathcal{A})$ the set of two sided ideals of an associative algebra $\mathcal{A}$. Following \cite{Lo1} in the non-super case,  we construct a map $ \bullet^{\dagger}: \mathfrak{id}(\mathcal{U}) \rightarrow \mathfrak{id}(\mathcal{U}(\ggg,e))$ and $\bullet_{\dagger}:\mathfrak{id}(\mathcal{U}) \rightarrow \mathfrak{id}(\mathcal{U}(\ggg,e))$. In the case of semisimple Lie algebra, those correspondences were used to prove the  existence of finite-/one- dimensional $\mathcal{U}(\ggg,e)$-modules.

In the case when $e$ is a regular nilpotent element in general linear Lie superalgebra  or is a minimal nilpotent element in $\ggg_{\bar{0}}$ for arbitrary basic  Lie superalgebra $\ggg$, the finite-dimensional representations of $\mathcal{U}(\ggg, e)$ were studied in \cite{BBG}, \cite{BG} and \cite{ZS2}. In \cite{BBG}, the authors proved that irreducible modules over the principal finite $W$-superalgebra $\mathcal{U}(\ggg, e)$  are finite dimensional and classified them by highest weight theory using the triangular decomposition of $\mathcal{U}(\ggg, e)$.
In \cite{PS1} the authors considered the finite $W$-algebra $\mathcal{U}(\ggg, e)$  for a basic Lie superalgebra and the queer Lie superalgebra $Q(n)$ associted with regular even nilpotent coadjoint orbits and proved that all irreducible representations of $\mathcal{U}(\ggg, e)$  are finite dimensional. They classified irreducible $\mathcal{U}(\ggg, e)$-modules for $Q(n)$ in \cite{PS3}. For general W-superalgebras of type $Q(n)$, see \cite{PS2}.
All the representations in \textit{loc.cit} were constructed somehow from an explicit presentation of $\mathcal{U}(\ggg, e)$.
In the present work we study finite-dimensional representations of $\mathcal{U}(\ggg, e)$ by the above mentioned conceptional approach. In the last section we construct a series of finite-dimensional modules via the above correspondences (see Theorem \ref{thm5.3}).  On the other hand,  we show that for any primitive ideal $\mathcal{J} \subset \mathcal{U} $ supported on $\bar{\mathcal{O}}$,  the pre-image of $\mathcal{J}$ under the map $\bullet_\dag$ is finite. Recall that Lezcter establishes a bijection between primitive ideals of $\mathcal{U}(\ggg_{\bar{0}})$ and $\mathcal{U}$ (c.f. \cite{Le}).  For further information about the later, see \cite{CM}.  Thus we make a step forward to a classification of the finite-dimensional irreducible modules over $\mathcal{U}(\ggg, e)$.

\section*{Acknowledgement}
We thank Losev for the suggestion of realizing finite $W$-superalgebra via DW theorem.  A part of the work was accomplished during the second author's visit to Northeastern University, he is grateful to Losev for stimulating discussions and hospitality. We thank Yang Zeng for helpful
discussions. We would like to express our gratitude to referee for the comments. B.S. is partially supported by the NSF of China (Nos.11671138,11771279) and Shanghai Key Laboratory of PMMP (No. 13dz2260400).
 H.X. is  partially supported by the NSF of China (No.11801113), Fundamental Research Funds for the Central Universities (Nos.GK2240260002,GK211026024).

\section{Proofs of theorem \ref{EDWthm} and \ref{SQDWtm}} \label{secSDTHMS}

Let us recall the notion of Poisson bi-vector at first. Especially, $M_\chi$ denotes the maximal ideal of $S[\vvv]$ corresponding to $\chi$,  this is to say, $M_\chi$ is the ideal generated by $x-\chi(x)$ for $x\in \vvv$.
Identify  the  cotangent space of the formal super scheme $\Spec(\hat{A})$ at the close point $\chi$ with the super space $\vvv$.  For a Poisson bracket $\{\cdot,\cdot\}$ on $\hat{A}$, the \textit{ Poisson bi-vector} $\Pi \in \bigwedge^2 \vvv$ associated to it is  given by
\begin{equation}\label{bivectr eq}
\Pi(\text{d}f,\text{d}g):= \{f,g\}+M_\chi \in \mathbb{C}.
\end{equation}
If $\Pi$ is non-degenerate, we say that the Poisson bracket $\{  \cdot , \cdot \}$ is symplectic.

Assume that $\hat{A}$ is equipped with an even continuous Poisson bracket $\{ \cdot , \cdot \}$.
\begin{lemma}\label{non-super decom thm}
\begin{itemize}

\item[(\rmnum{1})]Suppose that $f,g \in \hat{A}{M}_\chi$ are even elements with $\{f,g\}=1+t$  for some $ t \in M_\chi$.
Then there exists $g' \in  \hat{A}{M}_\chi^{2} $ such that $\{f,g+g'\}=1$.

\item[(\rmnum{2})]
Suppose that $f,g \in \hat{A}{M}_\chi$ are even elements with $\{f,g\}=1$. Denote by $V_1$ the symplectic vector space spanned by $f$ and $g$ with symplectic form $\omega_1(f,g)=1$. Then we have the following isomorphism of Poisson algebras
$$\phi_1: S[[V_1]] \otimes \hat{B}_1 \rightarrow \hat{A} ; \quad  a \otimes b \mapsto ab,$$
where $\hat{B}_1=\ker(\ad_f) \cap \ker(\ad_g)$.
\end{itemize}
\end{lemma}

Before giving the proof we clarify the notations of $\deg(a)$, $\ad_{a}$ and $\ad(a)$.
Fix a $\mathbb{Z}_2$-homogenous basis  $x_1,x_2, \dots, x_{N}$  of $\vvv'=\{ x-\chi(x) \mid  x \in \vvv \} \subset S(\vvv)$.
Recall that $A=S[\vvv]$, $\hat{A}=S[\vvv_{\bar{0}}]^{\wedge_{\chi}}\otimes \bigwedge (\vvv_{\bar{1}})$. Let $M_{\chi}$ be the  maximal ideal of $\hat{A}$ generated by $x_1,x_2, \dots, x_{N}$. For a super monomial  $x_1^{i_1}\cdots x_{N}^{i_{N}} \in A $, its order is defined by $\sum_{k=1}^{N} i_k$.
For $a \in \hat{A}$,  write $a=\sum_{i}^{\infty} a_i$  where $a_i$ is the sum of the monomials with order $i$. The \textit{degree} $\deg(a)$ of $a$ is defined to be the minimal $i$ such that $a_i \neq 0$. For any  $a \in A$, by $\ad_{a}$ we mean the adjoint operator
$\bullet \mapsto \{a, \bullet\}$ on $A$; for an associative algebra $(\mathcal{A},\circ)$ and $ a \in \mathcal{A}$, we denote by $\ad(a)$ the super commutator operator $\bullet \mapsto [a,\bullet]$.

Now we are ready to prove Lemma \ref{non-super decom thm}.
\begin{proof}
(\rmnum{1}) Suppose $\deg(t)=r$. By the notation $o(t)$, we denote an element with higher degree than $t$.
Set
$$g_1= g+\sum_{i=1}^{r+1} (-1)^i \frac{1}{i!}g^i \ad^{i-1}_f(t).$$
Since
$$\{f ,g^i \ad^{i-1}_f(t) \}=ig^{i-1} \ad^{i-1}_f(t)+ g^i \ad^{i}_f(t)+ o(t)  \text{ for $i=1,2,3, \dots, r$,} $$
we have  $\{f , g_1\}=1+t_1$  for some $t_1$ with $\deg(t_1)>r$.
By the same procedure we obtain
$g_k$, $t_k$  with $\deg (g_k) >\deg (g_{k-1})$, $\deg (t_k) >\deg (t_{k-1})$ and
$\{f, g_{k}\}=1+t_{k}$
for all $k=2,3, 4,\ldots $.
Since $\deg(g^i \ad^{i}_f(t)) \geq \deg(t)$, the series $\{g_{k}\}_{k=1}^{k=\infty}$ converges in $\hat{A}$.
Set $g'=\lim_{k\rightarrow \infty} g_k-g$. It is easy to check that $\deg(g') \geq 2$ and $\{f,g+g'\}=1$. Thus (\rmnum{1}) follows.

(ii) For any $a \in \hat{A}$ observe that
$$a-\sum_{0 \leq i+j}^{\infty}(-1)^{i+j} \frac{1}{i!j!}g^i f^j\ad^{i}_f \circ \ad^{j}_g(a) \in \hat{B}_1.$$
Making the same observation as that for  all $\ad^{i}_f \circ \ad^{j}_g(a)$s in the above, we can see that  $\phi_1$ is surjective.
For $\sum_{i,j\geq 1}g^i f^j \otimes a_{i,j} \in \ker(\phi_1)$ i.e.
$$\sum_{i,j\geq 0}g^i f^j \cdot a_{i,j}=0$$
acting on both sides by  $\ad_f^k \circ \ad_g^l$, we have $a_{l,k}=0$. So $\phi_1$ is also injective.
\end{proof}

The following is an analogue of Lemma \ref{non-super decom thm}  for odd elements.
\begin{lemma}\label{super decom thm}
\begin{itemize}
\item[(\rmnum{1})] Suppose that there are odd elements $f,g \in \hat{A}M_{\chi}$ with $\{ f,g\}=1+t$ for some $t \in M_{\chi}$.  Then there exists an odd  $h \in \hat{A}$ such that $\{ h, h \}=1$.
\item[(\rmnum{2})] For an odd $h \in \hat{A}$  with $\{h,h\}=1$, we have  following isomorphism
$$ \phi:  \mathbb{C}h \otimes \hat{B}_{1} \longrightarrow \hat{A} ; \quad  a\otimes b \mapsto ab,$$
of Poisson algebras. Here $\hat{B}_{1}=\ker \ad_h$.
\end{itemize}
\end{lemma}

\begin{proof}

(\rmnum{1}) If either $\{f, f \}$ or $\{g, g\}$, say $\{f, f \}$ is invertible.
We write $\{f, f \}=1+t_1$ for some $t_1 \in M_\chi$ and set $h=\frac{1}{\sqrt{1+t_1}}f$ ( here we interpret $\sqrt{1 + \bullet}$  as its Taylor expansion about $\bullet$ in $\hat{A}$ ). In the other case (namely $\{f, f \}$ and $\{g, g \} \in M_\chi$), we write $\{f+g, f+g \}=1+t_1$ for some  $t_1 \in M_\chi$ and set $h=\frac{1}{\sqrt{1+t_1}}(f+g)$. In both cases, we have  $\{h , h\}=1$.

For (\rmnum{2}) we only need to prove that the homomorphism $\phi$ is bijective.
Suppose  $f_1 +f_2h=0$  where $f_1,f_2 \in \hat{B}_{1}$.
Acting on both sides by $\ad_h$ we have $f_2=0$ and hence $f_1=0$. So the homomorphism $ \phi$ is injective.
On the other hand we have
$$\{h,hf\}=f- h\{h,f\} \mbox{ for all $f \in \hat{A}$}.$$

It follows form the definition  of even Poisson bracket that  $\{h,hf\}, \{h,f\} \in \ker(\ad_h)$. So $\phi$ is surjective.
\end{proof}

Now we ready to give

\noindent\textit{ Proof of  Theorem \ref{EDWthm}}.
We process by induction on $\dim(V)=(2n| m)$. Suppose that $m=0$. Fix a $\mathbb{C}^\times$-homogenous basis  $x_1,x_2, \dots, x_{N}$  of $\vvv'$.  Since the $\mathbb{C}^{\times}$-action is even, those basis are also $\mathbb{Z}_2$-homogenous.
By the definition of  Poisson bi-vector, there exist even $x_{i},x_{j}$ with $\{ x_{i},x_{j}\}=1+t$ for some $ t \in M_\chi$.
Thus we have already find even homogenous $f,g \in \hat{A}M_\chi$ with $\{f,g \}=1+t$ for some $ t \in M_\chi$.
We can  do $\mathbb{C}^{\times}$-homogenously  in the proof of Lemma \ref{non-super decom thm} and obtain new $\mathbb{C}^{\times}$-homogenous even elements $f,g$ with $\{f,g \}=1$. By Lemma \ref{non-super decom thm}, we have the following $\mathbb{C}^{\times}$-equivariant isomorphism of Poisson algebras
$$\phi_1: S[[V_1]] \otimes \hat{B}_1 \rightarrow \hat{A} ;  a \otimes b \mapsto ab. $$
Thus we can complete the proof in this case by induction on $\dim(V)$.

Suppose that the theorem holds for all $p< m$.  We may find odd $\mathbb{C}^{\times}$-homogenous $ f,g \in \hat{A}M_{\chi}$ such that $\{f, g\}=1+t$ with $t \in M_\chi$.

\textbf{Case 1:} Either $\{f, f \}$ or $\{g, g\}$ is invertible.\\
By the same argument as in case 1 of Lemma \ref{super decom thm} we may find a $\mathbb{C}^{\times}$-homogenous odd $h$ with $\{h, h\}=1$. Thus by Lemma \ref{super decom thm} (\rmnum{2}), we have the following $\mathbb{C}^{\times}$-equivariant isomorphism of Poisson algebras
$$ \phi_1 : \mathbb{C}h  \otimes  \hat{B}_{1} \longrightarrow \hat{A} ; \quad  a\otimes b \mapsto ab.$$
Now the theorem follows by induction.

\textbf{Case 2:}$\{f, f \},\{g, g \} \in M_\chi$.\\
We write
$$\{f,g \}=1+t,\{s_1,s_1 \}=t_1,\{g,g\}=t_2 \text{\quad with $t, t_1, t_2 \in M_\chi$ }. $$
Set $f_1=f-gt_1/2$. Then we have
$$\{f_1,g_1\}=1+t^{(1)},\{f_1,f_1\}=1+t_1^{(1)}, \{g,g\} =1+t_2\text{\quad with $t^{(1)}, t_1^{(1)}, t_2 \in M_\chi$}. $$
It is straightforward to check that  $\text{deg}(t_1^{(1)})> \text{deg}(t_1)$. Construct $ f_k,t_1^{(k)}$ in the same way  for  $k=2,3,4, \ldots $. Take $f_{\infty}=\lim_{k\rightarrow\infty} f_{k}$. We do the same procedure to the initial pair $(f_{\infty} ,g)$ and set $g_{\infty}=\lim_{k\rightarrow\infty} g_{k}$.
Now replace the $f$ and $g$ by $f_{\infty}$ and $g_{\infty}$ respectively. The new odd elements $f,g$ are $\mathbb{C}^{\times}$-homogenous and  satisfy
$$\{f,g \}=1+q,\{f,f\}=0,\{g, g \}=0 \text{\quad with $q \in M_\chi$}.$$
Replacing $f$ (\textit{resp.} $g$) by $\frac{f}{\sqrt{1+q}}$ (\textit{resp.} $\frac{g}{\sqrt{1+q}}$) for $i=1,2$, we may assume that  $q=0$ in the above equality. From the construction, we see that $f$ and $g$ are $\mathbb{C}^{\times}$-homogenous. For $h^{+}=f+g, h^{-}=f-g$, we have $\{ h^{+},h^{+} \} =1, \{ h^{-},h^{-} \}=1,\{ h^{-},h^{+} \}=0$. Using Lemma \ref{super decom thm} twice, we have  $\mathbb{C}^{\times}$-equivariant isomorphism
$$ \phi : \hat{B}_{1}\otimes \bigwedge (V_1) \longrightarrow \hat{A} ; \quad  a\otimes b \mapsto ab$$
of Poisson algebras. Here  $V_1=\mathbb{C}\langle h^{+},h^{-}\rangle =\mathbb{C}\langle f,g \rangle$ and $\hat{B}_{1}=\ker(\ad_f)\cap \ker(\ad_g)$. Now we can complete the proof  by induction. $\hfill\Box$

\begin{remark}
Specially if $V = \mathcal{T}_\chi^{*}(\spec(A)) (= \vvv)$, the theorem is called formal Darboux theorem.
Our proof  still valid in case of smooth super manifolds and is more direct and simpler than the one outlined in \cite{Kos}.
\end{remark}

\textit{Proof of Theorem \ref{SQDWtm}.}

We prove the theorem by induction on $m$. If $m=0$, choose  $\mathbb{C}^{\times}$-homogenous even elements $f,g \in M_{\chi}$ such that $\{f,g\}=1$. So we may write
$$[f,g]=\hbar^{2}+a_2\hbar^3+ o(\hbar^3)$$
Where $a_2 \in \hat{A}$, $o(\hbar^3)$ means some element in $\hat{A}[[\hbar]]\hbar^4$ and $[f,g]$ is by definition the super commutator $f\ast g -(-1)^{|f||g|}g\ast f$ for $f,g \in \hat{A}.$
Set
$$g_1=g-(\sum_{i=1}^{\infty}(-1)^{i}g^{i}\ad^{i-1}(f)(a_2))\hbar^2.$$
We have
$$[f,g_1]=\hbar^2+a_3\hbar^4+ o(\hbar^4).$$

Define $ g_k $  iteratively  for $k=2,3, \dots$. Replacing $g$ by $\LIM g_k$, we have$[f,g]=\hbar^2$.
By the similar argument in the proof of Lemma \ref{non-super decom thm} (\rmnum{2}), we have the  following isomorphism of
quantum algebras
\begin{equation}
\Phi_{1,\hbar}: \mathbf{A}_{\hbar}^{\wedge_{0}}(V_1) \otimes \hat{B}_{1,\hbar} \rightarrow \hat{A}[[\hbar]];  a \otimes b \mapsto ab.
\end{equation}
where $V_1=\mathbb{C}\langle f,g \rangle$ and $\hat{B}_{1,\hbar}=\ker(\ad(f)) \cap \ker(\ad(g))$.

For any $a \in \hat{A}[[\hbar]]$ we have
\begin{equation}\label{eq quan decom}
a-\sum_{i,j}(-1)^{i+j}\frac{1}{i!j!}f^jg^i\ad(f)^{i}\ad^j(g)(a)\hbar^{-2(i+j)} \in \ker(\ad(f)) \cap \ker(\ad(g))
\end{equation}

Applying same procedure to  $\ad(f)^{i}\ad^j(g)(a)$ for all $i,j \in \mathbb{N}$, we have that $\Phi_{1,\hbar}$ is surjective.
By a similar argument as in Lemma (\rmnum{2}) \ref{super decom thm}, we show that $\Phi_{1,\hbar}$ is injective. Thus in the case of $m=0$, the proof is completed by induction on $n$.

Now suppose $m>0$.\\
\textbf{Case 1}: There exists an odd $\mathbb{C^{\times}}$-homogenous $f$ such that $\{f, f\}=1+t$, with $t \in M_\chi$.
By Lemma \ref{super decom thm} we may assume that $\{f, f\}=1$ and have
$$ [f,f]=\hbar^2+a_2\hbar^3 + o(\hbar^3).$$
Set $f_1=f-\frac{1}{2}f\ast a_2\hbar$. Since $[f,[f,f]]=0$, we have $[f,a_2]=0$. This implies
$$[f_1,f_1]=\hbar^2+a_3\hbar^4 + o(\hbar^4).$$

Define $f_k$ for $k=2,3,4 \dots$ iteratively and replace  $f$ by $\LIM f_k$. For the new $f$, we have $[f,f]=\hbar^2$.
By the same argument as in non-quantum case, we have the following $\mathbb{C}^{\times}$-equivariant homomorphism of quantum algebras
$$\Phi_{1,\hbar}: \hat{B}_{1,\hbar}\otimes \bigwedge (\mathbb{C}f)[[\hbar]] \longrightarrow \hat{A}[[\hbar]] ; \quad  a\otimes b \mapsto a\ast b.$$
Here $\bar{B}_{1,\hbar}=\ker(\ad(f))$.

\textbf{Case 2}: Otherwise, by  case 2 in the proof  of Theorem \ref{EDWthm}, we may find $\mathbb{C}^{\times}$-homogenous $f,g \in \abo$  with
$$\{f,g\}=1,\{f,f\}=0,\{f, g \}=0$$
Using similar arguments as in the proof of Theorem \ref{EDWthm} and the case 1, we may construct new $f,g \in A[[\hbar]]$
such that
$$[f,g ]=\hbar^2, [f,f]=0, [g, g]=0.$$

Thus  we may construct a $\mathbb{C}^{\times}$-equivariant isomorphism of quantum algebras as  case 2 in the proof of Theorem \ref{EDWthm}. The proof is completed by induction on $\dim(V)=(2n| m)$. $\hfill\Box$

\section{Realization of W-superalgebras via Darboux-Weinsten theorem}\label{sec realzation of W}

In the next two subsections we recall the results from \cite{Lo1} on the Rees algebra of the filtered  associative algebras and completion of $\mathcal{U}$. Although those results are given in non-super cases, they are still hold in the super cases.

\subsection{Rees algebras}

Let $\mathcal{A}$ be an associative algebra with an increasing filtration $\F_i\mathcal{A},$ $i \in \mathbb{Z}$ such that $\bigcup_{i \in \mathbb{Z}}\F_i\mathcal{A}=\mathcal{A}, \bigcap_{i \in \mathbb{Z}}\F_i\mathcal{A}=0$.
Let $\mathbf{R}_{\hbar}(\mathcal{A})=\bigoplus_{i\in \mathbb{Z}}\F_i\mathcal{A}\hbar^{i}$ and view it as  a subalgebra of $\mathcal{A}[\hbar,\hbar^{-1}]$.
The algebra $\mathbf{R}_{\hbar}(\mathcal{A})$  is called Rees algebra of $\mathcal{A}$. For an ideal $\mathcal{I} \in \mathcal{A}$, define $\mathbf{R}_\hbar(\mathcal{I})=\bigoplus_{i \in \mathbb{Z}}(\mathcal{I} \cap \F_i\mathcal{A})\hbar^i$.

\begin{defn}(\cite[Definition 3.2.1]{Lo1})
An ideal $I$ in $\mathbb{C}[\hbar]$-algebra $B$ is called $\hbar$-saturated if $I=\hbar^{-1}I\cap B$.
\end{defn}

\begin{prop}(\cite[Proposition 3.2.2]{Lo1})
The map  $\mathcal{I} \mapsto  \mathbf{R}_\hbar(\mathcal{I})$ establishes a bijection between the set $\mathfrak{id}(\mathcal{A})$ and the set of $\hbar$-saturated ideals of $\mathbf{R}_{\hbar}(\mathcal{A})$.
\end{prop}

Let $\mathfrak{v}=\mathfrak{v}_{\bar{0}}+\mathfrak{v}_{\bar{1}}$ be a superspace with a $\mathbb{C}^\times$-action. Suppose that there is a
Poisson bracket $\{\cdot ,\cdot \}$ on $\hat{A}=S[[\mathfrak{v}]]$  and  $k \in \bbz$ with  $\{t\cdot f, t\cdot g  \}=t^k\{f,g\}$ for all $f,g \in \hat{A}$. Let $(\hat{A}_{\hbar}=S[[\mathfrak{v},\hbar]], \ast)$ be a $\mathbb{C}^\times$-equivariant quantization of $\hat{A}$ as in \S\ref{sec 1.1}. Assume that $A_{\hbar}=S[\mathfrak{v},\hbar]$ form a  subalgebra of quantum algebra $S[[\mathfrak{v},\hbar]]$, i.e. $A_{\hbar}$ is closed under the product $\ast$.

For $f,g \in \hat{A}$ write
$$f\ast g =\sum_{i}D_i(f,g)\hbar^{i},$$
where $D_i: \hat{A} \otimes \hat{A} \rightarrow \hat{A}$ is a linear map determined by the product $\ast$ in the quantum algebra $\hat{A}[[\hbar]]$.

There is a product $\circ$ on $\hat{A}$ given by
\begin{equation}\label{mul from qum}
f\circ g =\sum_{i}D_i(f,g).
\end{equation}
We denote by $\hat{\mathcal{A}}$ the algebra from it. Let $\mathcal{A} \subset \hat{\mathcal{A}}$ be the subalgebra from $A:=S(\vvv)$.

There is an action of $\mathbb{C}^{\times}$ on
the Rees algebra $\mathbf{R}_{\hbar}(\mathcal{A})$ by
\begin{equation}\label{action Rees}
t\cdot a_i\hbar^{i} =t^{i}a_i\hbar^{i}
\end{equation}
for all $a_i\hbar^{i}\in \F_{i}(\mathcal{A})\hbar^{i}$ and $t \in \mathbb{C}^{\times}$.

Take $\mathcal{I} \in \mathfrak{id}(\mathcal{A})$ and set $I=\gr(\mathcal{I})$. Let $\overline{\mathcal{I}}_{\hbar}$ be the closure of $\mathbf{R}_\hbar(\mathcal{I})$ with respect to the $M_{0,\hbar}$-adic topology on $\hat{A}_{\hbar}$.
\begin{prop} (\cite[Proposition 3.2.10]{Lo1}) \label{prop3.4}
\begin{itemize}
\item[(1)] The ideal $\overline{\mathcal{I}}_{\hbar}$ is $\hbar$-saturated.
\item[(2)] The classical part of $\overline{\mathcal{I}}_{\hbar}$ coincides with the closure  $\hat{I}$ of $I$ in $S[[\mathfrak{v}]]$.
Here the classical part  of an ideal of $\hat{A}_\hbar$ means its image under the  projection $\hat{A}_\hbar \rightarrow \hat{A}$ by specializing $\hbar$ to $0$.
\item[(3)] Suppose that the grading on $\mathfrak{v}$ is positive. Let $\mathcal{I}_\hbar$ be a closed $\hbar$-saturated, $\mathbb{C}^\times$-stable ideal of $\overline{\mathbf{R}_\hbar(A)}$, then there exists a unique $\mathcal{I} \in \mathfrak{id}(\mathcal{A})$ such that $\mathcal{I}_\hbar \cap A[\hbar]=\mathbf{R}_\hbar(\mathcal{I})$.
\end{itemize}
\end{prop}

\subsection{Completions of $\mathcal{U}$} \label{sec: 3.2}
Let $\mathcal{U}$ be the universal enveloping algebra of the Lie superalgebra $\ggg$. Let $\ggg=\oplus_i\ggg(i)$ be the good $\mathbb{Z}$-grading in \S\ref{sec introduction}. The \textit{Kazhdan action} of $\mathbb{C}^{\times}$ on $\ggg$ is given by $t\cdot y=t^{i+2}y$ for all $t \in \mathbb{C}^{\times}$ and $y \in \ggg(i)$. Since $\chi \in \ggg^{*}$ is an element of degree $-2$, $\chi$ is a fixed point of the induced $\mathbb{C}^{\times}$-action on $\ggg^{*}$. Extend the Kazhdan action to $\mathcal{U}$ canonically.

\subsubsection{} Let $A_{\hbar}=S[\ggg][\hbar]$ (see Example \ref{ex quan} for the definition), $A=S[\ggg]$ and $\pi:A_{\hbar} \rightarrow A$  be the algebra homomorphism given by taking $\hbar=0$. Let $M_{\chi}$ be the maximal ideal of $A=S[\ggg]$ corresponding to the point $\chi \in \gbz^{*}$ and  $A^{\wedge_{\chi}}$ be the completion of $A$ with respect to the ideal $M_{\chi}$. Set
$$(A_{\hbar})^{\wedge_{\chi}}=\lim_{ \leftarrow }A_{\hbar}/ M_{\chi,\hbar}^i,$$
where $ M_{\chi,\hbar}$ is the two-sided  ideal of $A_{\hbar}$ generated by $\pi^{-1}(M_{\chi})$.
We extend the Kazadan action on $A_{\hbar}$ to $(A_{\hbar})^{\wedge_{\chi}}$ by $t\cdot \hbar =t\hbar$.
It is easy to check that $(A_{\hbar})^{\wedge_{\chi}}$ is a $\mathbb{C}^{\times}$-equivariant quantization of the Poisson algebra $A^{\wedge_{\chi}}$. We simplify the notations $A^{\wedge_{\chi}}$ and $A^{\wedge_{\chi}}_{\hbar}$ by $\hat{A}$ and by $\hat{A}_{\hbar}$, respectively.

\subsubsection{}  Set
\begin{equation}\label{g_e}
\tilde{\ggg}_e=\begin{cases}
 \ggg_e \quad \mbox{if $\dim(\ggg(-1)_{\bar{1}})$ is even}; \\
\ggg_e +\mathbb{C}\Theta  \quad \mbox{if $\dim(\ggg(-1)_{\bar{1}})$ is odd }.
 \end{cases}
 \end{equation}
Choose a homogenous (with respect to the Kazhdan action) basis
$$v_{\pm1},v_{\pm2},\dots,v_{\pm n},v_{2n+1}, \dots, v_{2n+m}$$
of $\ggg'$ such that $v_1,v_2, \dots, v_n$ (\textit{resp}. $v_{-1},v_{-2}, \dots, v_{-n}$) form a basis of $\mmm'$  (\textit{resp}. $(\mmm')^*$) and
$v_{2n+1}, \dots,  v_{2n+m}$ form a basis of $\tilde{\ggg}_{e}'$. Let $d_i$ denote the degree of $v_i$, $i\in\{\pm1,\pm2,\dots,\pm n; 2n+1,\dots,2n+m\}$.
Denote by $\mathcal{A}^{\heartsuit}$ the subalgebra of $\hat{A}$ consisting of elements $\sum_{i<n}f_i$ where $f_i$ is a homogenous power
series with degree $i$ (with respect to the Kazhdan grading). Let $\mathcal{I}^{\heartsuit}(k)$ be the ideal of $\mathcal{A}^{\heartsuit}$ consisting of all those $a$ satisfying that any monomial in $a$ has the form $v_{i_{1}}\circ \cdots \circ v_{i_{l}}$ with $v_{i_{l-k+1}} \in \mmm'$.

We need to introduce some new notion before giving the next lemmas.  Generally, for an associative algebra $\mathcal{A}$ with $\mathbb{C}^{\times}$-action,
by $\mathcal{A}_{\mathbb{C}^{\times}\text{-fin}}$ we mean the $\mathbb{C}^{\times}$-local finite part of $\mathcal{A}$, i.e. the sum of all finite-dimensional $\mathbb{C}^{\times}$-stable subspaces $\underline{\mathcal{A}} \subset \mathcal{A}$.

\begin{lemma}\label{lemma3.4}
\begin{itemize}
\item[(\rmnum{1})]
The map
$$A[\hbar] \rightarrow \mathcal{A}[\hbar,\hbar^{-1}];\; \sum_{i,j}f_i\hbar^j \mapsto \sum_{i,j}f_i\hbar^{i+j}$$ is a $\mathbb{C}^{\times}$-equivariant monomorphism of $\mathbb{C}[\hbar]$-algebras. Its image coincides with $\mathbf{R}_{\hbar}(\mathcal{A})$.
\item[(\rmnum{2})]
The map
$$S[[\ggg,\hbar]]_{\mathbb{C}^{\times}\text{-fin}} \rightarrow \mathcal{A}^{\heartsuit}[\hbar,\hbar^{-1}];  \sum_{i,j}f_i\hbar^j \mapsto \sum_{i,j}f_i\hbar^{i+j}$$
is a $\mathbb{C}^{\times}$-equivariant monomorphism of $\mathbb{C}[\hbar]$-algebras. Its image coincides with $\mathbf{R}_{\hbar}(\mathcal{A}^{\heartsuit})$.
\end{itemize}
Here $f_i$  is a $\mathbb{C}^{\times}$-homogenous element in $A$ with $t\cdot f_i =t^if_i$ for all $t \in \mathbb{C}^{\times}$.
The $\mathbb{C}^{\times}$-action on the images of the above two maps are give by \eqref{action Rees}.
\end{lemma}

\begin{lemma}\label{equ of ideals}
Let $\tilde{\mathfrak{m}}$ be a subspace in $\mathcal{A}^{\heartsuit}$, which  has a basis in the form of
$v_{i}+u_{i}, i=1,\cdots m$, with $u_i \in \text{F}_{d_{i}}(\mathcal{A}^{\heartsuit}) \cup \mathfrak{g}'^{2}\mathcal{A}^{\heartsuit}+\text{F}_{d_{i}-2}(\mathcal{A}^{\heartsuit})  $. Then
$\mathcal{A}^{\heartsuit}\tilde{\mathfrak{m}} \in \mathcal{I}^{\heartsuit}(1)$ and for any $k \in \mathbb{N}$ there is
$l \in \mathbb{N}$ such that $\mathcal{A}^{\heartsuit}\tilde{\mathfrak{m}}^{l} \in \mathcal{I}^{\heartsuit}(k)$.
\end{lemma}

\begin{lemma}\label{compatible lemma}
For the universal algebra $ \mathcal{U} \subset \mathcal{A}^{\heartsuit}$, the systems $\mathcal{U}\mmm'^{k}$ and $\mathcal{I}(k):=\mathcal{I}^{\heartsuit} (k)\cap \mathcal{U} $ are compatible, i.e. for any $k \in \mathbb{N}$ there exist
$k_1,k_2$ such that $\mathcal{I}(k_1)\subset \mathcal{A}\mmm'^k$, $\mathcal{A}\mmm'^{k_2} \subset \mathcal{I}(k)$.
\end{lemma}

The above-mentioned three lemmas are super versions of Lemmas 3.2.5, 3.2.8 and 3.2.9 in \cite{Lo1}, which will be used later. We omit the proofs of them, which are exactly the same as in the non-super case.

\subsection{Construction of $\mathscr{W}_\chi$}\label{subsec con of w}
Recall that the Poisson bivector $\Pi$ on the closed point $\chi$ of $\text{Spec}(\hat{A})$ is given by $(x,y)\mapsto \chi([x,y])$ for all $x,y \in \ggg$.
It is easy to check that $\Pi$  is non-degenerate on $[\ggg,f] \subset \mathcal{T}^*_\chi (\text{Spec}(\hat{A}))$.
Thus
$\Pi$ is also non-degenerate on $V=\bar{\mathfrak{m}}\oplus \bar{\mathfrak{m}}^*$ (see \S\ref{sec: completion 0} for the notations).
Since there is a $\mathbb{C}^{\times}$-stable decomposition $\ggg=V \oplus \tilde{\ggg}_{e}$  with respect to the Kazhdan action, we have $\hat{B}_{\hbar}=S[[\tilde{\ggg}_{e},\hbar]]$ as vector superspaces.
By Theorem \ref{SQDWtm}, we have the following isomorphism
\begin{equation}\label{eqdecom}
 \Phi_{\hbar}: \mathbf{A}_\hbar^{\wedge_{0}}(V)\hat{\otimes}_{\mathbb{C}[[\hbar]]}\hat{B}_{\hbar} \rightarrow \hat{A}_{\hbar}
\end{equation}
of quantum algebras.

Now we set
\begin{equation}
\mathscr{W}_{\chi}=\frac{(\hat{B}_{\hbar})_{\mathbb{C}^{\times}\text{-fin}}}{(\hbar-1)(\hat{B}_{\hbar})_{\mathbb{C}^{\times}\text{-fin}}}.
\end{equation}
We will show that $\mathscr{W}_{\chi}$ is isomorphic to the $W$-superalgebra $\mathcal{U}(\ggg,e)$ . Let us discuss the relation between $\mathscr{W}_{\chi}$ and the transverse slices of
super nilpotent orbit $G\cdot \chi$ before doing that.
Since there is a $\mathbb{C}^{\times}$-stable decomposition $\ggg=V \oplus \tilde{\ggg}_{e}$  with respect to the Kazhdan action, we have $\hat{B}_{\hbar}=S[[\tilde{\ggg}_{e},\hbar]]$ as vector superspaces.  Since all $\mathbb{C}^{\times}$-homogenous elements in $\tilde{\ggg}_{e}$ have positive \textit{Kazhdan} grading, we have $(\hat{B}_{\hbar})_{\mathbb{C}^{\times}\text{-fin}}=S[\tilde{\ggg}_{e},\hbar]$ as vector superspaces.
The product $\ast$ on $(\hat{B}_{\hbar})_{\mathbb{C}^{\times}\text{-fin}}$ gives $S[\tilde{\ggg}_{e},\hbar]$ a quantum algebra structure.
Thus it is clear from the construction that $\mathscr{W}_{\chi}$ is a filtered algebra with $\text{gr}(\mathscr{W}_{\chi})=S[{\ggg}_{e}]$, if $\dim \ggg(-1)_\bo$ is even; and $\text{gr}(\mathscr{W}_{\chi})=S[{\ggg}_{e}]\otimes \bbc[\Theta]$ if $\dim \ggg(-1)_\bo$ is odd, where $\bbc[\Theta]$ is the exterior algebra generated by one element $\Theta$.
This is a PBW theorem for $\mathscr{W}_{\chi}$ (compare with theorem 0.1 in \cite{ZS1}).
This  provides $S[\tilde{\ggg}_{e}]$ with an even Poisson bracket, and $\mathscr{W}_{\chi}$ is a filtered quantization of $\mathscr{W}_{\chi}$ of this Poisson algebra (c.f \cite[\S1.3]{CG}).

We will write $\mathscr{W}_{\chi}$  as $\mathscr{W}$ in the sequence for simplicity.

\subsection{The realization of $\mathcal{U}(\ggg,e)$ via $\mathscr{W}$}
\begin{lemma}\label{decom of qun unive le}
Let $v_{i}$, $i=\pm 1,\pm 2, \ldots, \pm n$ be a basis of the symplectic space $V \subset \ggg'$ with $\langle v_{i}, v_{-j}\rangle =\delta_{i,j}$.
Then  $ v_i- \Phi_{\hbar}^{-1}(v_{i}) \in F_{d_{i}}(\hat{A}_{\hbar}) \cap \mathfrak{g'}^{2}\hat{A}_{\hbar}$ for  $i=\pm1,\pm2, \ldots, \pm n$. We also have $\Phi_{\hbar}^{-1}(x)-x \in V \hat{A}_{\hbar}$ for all $x \in \ggg'$.
\end{lemma}

\begin{proof}
Note that $[v_1,v_{-1}]=\hbar^{2}+\ggg'\hat{A}_{\hbar}$. We can modify $v_1$, $v_{-1}$ as in the proof of Theorem \ref{SQDWtm} and
get new elements $\bar{v}_1$, $\bar{v}_{-1}$ with $[\bar{v}_1,\bar{v}_{-1}]=\hbar^{2}$ and $\bar{v}_i=v_i+u_i$ for some $u_i  \in \ggg'^2\hat{A}_{\hbar}$ for $i=1,-1$. Set $V_1=\mathbb{C}\langle \bar{v}_{1},\bar{v}_{-1}\rangle$,  $ \hat{B}_1[[\hbar]]=\ker(\ad(\bar{v}_{1})) \cap \ker(\ad(\bar{v}_{-1}))$. By Theorem \ref{SQDWtm} we have the following isomorphism
$$\Phi_{1,\hbar}: \begin{cases} S[[V_1]][[\hbar]] \otimes \hat{B}_1[[\hbar]] &\rightarrow \hat{A}_{\hbar},  \text{ \quad  if $v_1 \in \ggg'_{\bar{0}}$;}   \\
 \bigwedge(V_1)[[\hbar]] \otimes \hat{B}_1[[\hbar]] &\rightarrow \hat{A}_{\hbar}, \text{ \quad if $v_1 \in \ggg_{\bar{1}}$}
\end{cases}$$
of quantum algebras.

By the arguments in the proof of Theorem \ref{SQDWtm},  for $i= \pm2,\ldots, \pm n$, there exist $u_i^{(1)} \in \ggg'^2\hat{A}_{\hbar}$  such that $u_i^{(1)}+v_i \in \hat{B}_1[[\hbar]]$.
It is easy to check that
$$ [u_i^{(1)}+v_i , u_j^{(1)}+v_j]-\delta_{i,-j}\hbar^2  \in \ggg'\hat{A}_{\hbar}.$$
Now we can accomplish the proof of the first statement by induction. The second statement follows from the construction of the isomorphism $\Phi_{\hbar}$.
\end{proof}

  Denote by $\hat{\mathcal{A}}_{1}$  the associative algebra from the quantum algebra $\mathbf{A}_\hbar^{\wedge_{0}}(V)\hat{\otimes}_{\mathbb{C}[[\hbar]]}\hat{B}_{\hbar}$ as in the equation \eqref{mul from qum}.
Construct  $\mathcal{A}_1^{\heartsuit}$ in the same way as  $\mathcal{A}^{\heartsuit}$, which is actually a subalgebra of $\hat{\mathcal{A}}_1$ from $A_1:=S[V]\otimes \bbc[\tilde\ggg_e]$. Recall that  $\bar{\mmm}$ is  the Lagrangian subspace of $V$  spanned by $v_{-1}, \dots, v_{-n}$. Construct $\mathcal {I}_1^{\heartsuit}(k)$ from $\bar{\mmm}$ in the same way  as $\mathcal {I}^{\heartsuit}(k)$ from $\mmm'$ in \S\ref{sec: 3.2}.

\begin{theorem} \label{realzation of W}
$\mathscr{W}$ and $ \mathcal{U}(\ggg,e)$ are isomorphic as associative algebras.
\end{theorem}

The proof of the above theorem is basically somewhat as in the non-super case in \cite{Lo1}. We recall it for the reader's convenience.

\begin{proof}
By Lemma \ref{lemma3.4} we have the following isomorphisms
$$
(\hat{A}_\hbar)_{\mathbb{C}^{\times}\text{-fin}}\cong \mathbf{R}_{\hbar}(\mathcal{A}^{\heartsuit}); \;\sum_i f_i \hbar^{j}\mapsto f_i \hbar^{i+j}$$
and
$$(\mathbf{A}_\hbar^{\wedge_{0}}(V)\hat{\otimes}_{\mathbb{C}[[\hbar]]} \hat{B}_{\hbar})_{\mathbb{C}^{\times}\text{-fin}}\cong \mathbf{R}_{\hbar}(\mathcal{A}_{1}^{\heartsuit}); \; \sum_i f_i \hbar^{j}\mapsto f_i \hbar^{i+j}$$
of $\mathbb{C}[\hbar]$ algebras.
Thus restricting $\Phi_{\hbar}$ to the $\mathbb{C}^{\times}$-local finite part of $\hat{A}_\hbar$, we obtain an isomorphism $\Phi_{\hbar}^{\heartsuit}: \mathbf{R}_{\hbar}(\mathcal{A}_{1}^{\heartsuit})\rightarrow \mathbf{R}_{\hbar}(\mathcal{A}^{\heartsuit} ) $ of $\mathbb{C}[\hbar]$-algebras.

Taking $\hbar=1$, we have a $\mathbb{C}^{\times}$-equivariant  isomorphism $\Phi^{\heartsuit} :\mathcal{A}^{\heartsuit} \rightarrow \mathcal{A}_{1}^{\heartsuit}$ of algebras (with respect to  ``$\circ$"). Lemmas \ref{equ of ideals} and \ref{decom of qun unive le} imply $\Phi^{\heartsuit}(\mathcal{I}_1^{\heartsuit}(1)) \subset \mathcal{I}^{\heartsuit}(1)$ and $(\Phi^{\heartsuit})^{-1}(\mathcal{I}^\heartsuit(1)) \subset \mathcal{I}_1^{\heartsuit}(1)$. Hence $\Phi^{\heartsuit}(\mathcal{I}_1^{\heartsuit}(1))=\mathcal{I}^{\heartsuit}(1)$.

From the definition (\ref{def: u(g,e)}), it follows that  $\mathcal{U}(\mathfrak{g},e)\cong(\mathcal{A}^{\heartsuit}/\mathcal{I}^\heartsuit(1))^{\mathcal{I}^\heartsuit(1)}$.
By the construction of $\mathscr{W}$, we know that $(\mathcal{A}_1^{\heartsuit}/\mathcal{I}^\heartsuit_1(1))^{\mathcal{I}^\heartsuit_1(1)}\cong\mathscr{W}$.

Thus finally we have
$$\mathcal{U}(\mathfrak{g},e)\cong(\mathcal{A}^{\heartsuit}/\mathcal{I}^\heartsuit(1))^{\mathcal{I}^\heartsuit(1)}
\cong(\mathcal{A}_1^{\heartsuit}/\mathcal{I}^\heartsuit_1(1))^{\mathcal{I}^\heartsuit_1(1)}\cong\mathscr{W}.$$
\end{proof}

 \textit{ Proof of  Theorem \ref{realzn of w-alge thm}}.
By Lemma \ref{compatible lemma} and the proof of Theorem \ref{realzation of W} the systems $\Phi^{\heartsuit}(\mathbf{A}(V)\otimes \mathscr{W}/(\mathbf{A}(V)\otimes \mathscr{W})\bar{\mmm}^{k})$ and $\mathcal{U}\mmm'^{k}$ are compatible.
Thus the morphism $\Phi^{\heartsuit}$ can be extended to an isomorphism
$$\Phi: \mathbf{A}(V)_{\bar{\mmm}}^{\wedge}\otimes \mathscr{W}\cong\mathcal{U}^{\wedge}_{\mmm'}$$
of topological algebras. $\hfill\Box$

\vskip5pt
As a corollary to Theorem \ref{realzation of W} along with the arguments in \S\ref{subsec con of w}, we reobtain the following result which is a main result of \cite{ZS1}.
\begin{corollary} (\cite[Theorem 0.1]{ZS1}) Associated with a basic classical Lie superalgebra $\ggg$ and a nilpotent element $e\in \ggg_\bz$, the following statements hold
\begin{itemize}
\item[(1)] $\gr \mathcal{U}(\ggg,e)\cong S(\ggg_e)$ as $\bbc$-algebras when $\dim\ggg(-1)_{\bo}$ is even.

\item[(2)] $\gr \mathcal{U}(\ggg,e)\cong S(\ggg_e)\otimes \bbc[\Theta]$ as $\bbc$-algebras when $\dim\ggg(-1)_{\bo}$ is odd.
\end{itemize}
\end{corollary}

\begin{remark} In \cite{ZS3}, the authors introduced the so-called refined $W$-algebra $\mathcal{U}'(\ggg,e)$, which is by definition equal to $(\mathcal{U}\slash I_\chi)^{\ad \overline\mmm} $ with $\overline\mmm$ being either $\mmm$ itself or its one-dimensional extension (the occurrence is dependent on the parity of $\dim\ggg(-1)_\bo$, taking the former  if even, and  taking the latter  if odd).

 They showed that $\mathcal{U}'(\ggg,e)$ is a subalgebra of $\mathcal{U}(\ggg,e)$ and satisfies that $\gr\mathcal{U}'(\ggg,e)\cong S(\ggg_e)$. The arguments in the present paper are still available to this  refined version of the finite $W$-algebras, by some  correspondingly modifying in  the construction.
\end{remark}

\subsection{A relation between finite $W$-superalgebras and W-algebras}\label{rel of W's}

Let $\ggg=\ggg_{\bar{0}}\oplus \ggg_{\bar{1}}$ be a basic Lie superalgebra  and $e \in \ggg_{\bar{0}}$ be a nilpotent.
The construction of $\mathscr{W}$ in the present paper is inspired by its non-super counterpart given in \cite{Lo3}. We briefly describe it as below.

Let $\hat{A}_{0,\hbar}=S[\ggg_{\bar{0}}]_{\hbar}^{\wedge_\chi}$ be the $\mathbb{C}^{\times}$-equivariant quantization of $\hat{A}_{0}=S[\ggg_{\bar{0}}]^{\wedge_\chi}$
constructed in the same way as in \S\ref{sec: 3.2}. Restriction of the Poisson bivector on $V_{\bar{0}}=\mathcal{T}^*(G_{\bar{0}}\cdot \chi )$ is non-degenerate.  By the  eqDW theorem we have a $\bbc^{\times}$-equivariant isomorphism
$$\Phi_{0,\hbar}: \hat{A}_{0,\hbar} \rightarrow S[[V_{\bar{0}}]]_{\hbar} \otimes \hat{B}_{0,\hbar}.$$
of quantum algebras. Here the notations in the above isomorphism have the same meaning as in the previous sections. Now set
$$ \mathscr{W}_{0}=\frac{(\hat{B}_{0,\hbar})_{\mathbb{C}^{\times}\text{-fin}}}{(\hbar-1)(\hat{B}_{0,\hbar})_{\mathbb{C}^{\times}\text{-fin}}}.$$

By the same argument as in the proof of Theorem \ref{realzation of W}, we have that $\mathscr{W}_{0}$ is isomorphic to the finite $W$-algebra $\mathcal{U}(\ggg_{\bar{0}},e)$. The above fact is first proved in \cite{Lo2} relying on the main result of \cite{Lo2}, which is based on the Fedosov quantization and algebraic invariant theory. Those theories are not necessary to the arguments in the present paper.

Applying  Theorem \ref{SQDWtm} to $\hat{A}_{\hbar}=S[\ggg]^{\wedge_{\chi}}$ and $V_{\bar{0}}$, we have the following isomorphism
$$\hat{A}_{\hbar} \rightarrow S[[V_{\bar{0}}]]_{\hbar} \otimes \hat{C}_{\hbar}$$
of quantum algebras. Here $\hat{C}_{\hbar}$ is defined similarly to $\hat B_\hbar$ in  Theorem \ref{SQDWtm}. Denote by
$$\mathcal{A}_{\ddag}=\frac{(\hat{C}_{\hbar})_{\mathbb{C}^{\times}\text{-fin}}}{(\hat{C}_{\hbar})_{\mathbb{C}^{\times}\text{-fin}}(\hbar-1)}.$$

The following theorem relates the finite W-algebras and $W$-superalgebras.
\begin{theorem}\label{thm rel of W's}
\begin{itemize}

\item[(\rmnum{1})]We have an embedding $\mathscr{W}_{0} \hookrightarrow \mathcal{A}_{\ddag}$. The later is finitely generated over the former.
\item[(\rmnum{2})] There exists an isomorphism
$$ \Psi :  \mathrm{Cl}(V_{\bar{1}})\otimes \mathscr{W}\longrightarrow \mathcal{A}_{\ddag}$$
of associative  superalgebras. Here $V_{\bar{1}}$ is the odd part of $V=\bar{\mmm} \oplus \bar{\mmm}^*$, $\mathrm{Cl}(V_{\bar{1}})$ is the Clifford algebra from $(V_{\bar{1}},\chi([\cdot,\cdot]))$.
\end{itemize}
\end{theorem}

\begin{proof}

(\rmnum{1}) Let $v_{1},v_{-1} ,\ldots, v_{l},v_{-l}$ (\textit{resp.} $v_{l+1},v_{-l-1},\ldots, v_{n},v_{-n}$) be a basis of $V_{\bar{0}}$( \textit{resp} $V_{\bar{1}}$) with $\omega(v_i,v_{j})=\delta_{i,-j}$.  We construct $\bar{v}_{1}, \bar{v}_{-1},\ldots, \bar{v}_{n}, \bar{v}_{-n}$ as in the proof of Lemma \ref{decom of qun unive le} by the indicated order. Since we have $\bar{v}_{i} \in \hat{B}_{0,\hbar}$ for all $i=\pm 1 ,\ldots, \pm l$,  both of $\Phi_{\hbar}$ and $\Phi_{0,\hbar}$ may be given by those $\bar{v}_{i}$s as before. Thus we have the first statement of (\rmnum{1}).  Since $\bar{v}_{i} \in \hat{C}_{\hbar}$ for all $i=\pm (l+1),\dots, \pm n$, we have the following isomorphism
$$ \Psi_{\hbar}: \mathrm{Cl}(V_{\bar{1}})_{\hbar}\otimes \hat{B}_{\hbar} \longrightarrow \hat{C}_{\hbar}$$
of quantum algebras by a similar argument as in the proof of Theorem \ref{SQDWtm}. Now the statement (\rmnum{2}) follows from taking $\mathbb{C}^{\times}$-local finite part and specializing $\hbar$ to $1$. Note that $\gr(\mathcal{A}_{\ddag}$) is isomorphic to  $S[(\mathfrak{g}_{\bar{0}})_e] \otimes \bigwedge(\mathfrak{g}_{\bar{1}})$ and $\gr(\mathscr{W}_{0})$ is isomorphic to  $S[(\mathfrak{g}_{\bar{0}})_e]$. Here $(\mathfrak{g}_{\bar{0}})_e$ is the centralizer of $e$ in $\mathfrak{g}_{\bar{0}}$. Thus the
last statement of (\rmnum{1}) follows.
\end{proof}

\begin{corollary} \label{cor 3.12}
The finite $W$-superalgebra $\mathscr{W}$ has finite-dimensional representations.
For a regular nilpotent element  $e \in \mathfrak{g}_{\bar{0}}$,  any irreducible representation of $\mathscr{W}$ is finite-dimensional.
\end{corollary}

\begin{proof}
The first statement follows from the fact that $\mathscr{W}$ has finite dimensional modules.
Since $\mathscr{W}_{0}$ is isomorphic to the center of $\mathcal{U}(\mathfrak{g}_{\bar{0}})$, the second statement of Theorem \ref{thm rel of W's} (\rmnum{1}) implies that of the corollary.
\end{proof}

\begin{remark} \label{rem: clue to finite mod}
The statement (\rmnum{2}) of  Theorem \ref{thm rel of W's} provides a bijection between the sets of finite-dimensional irreducible modules of $\mathcal{A}_{\ddag}$ and $\mathscr{W}$. For some special $e \in \ggg$, we will give a classification of $\text{Irr}^{\text{fin}}(\mathscr{W})$ by using such a correspondence in a forthcoming work \cite{SXZ}.
\end{remark}

\section{On representations of $\mathscr{W}$}

\subsection{Skryabin's equivalence }

We say that a $\mathcal{U}$ module is  Whittaker if $x$ acts locally nilpotent on $M$ for all $x \in \mmm'$.
 For $Q_{\chi}=\mathcal{U}/\mathcal{U}\mmm'$, we have the following functors between the category $\mathcal{C}$ of Whittaker $\mathcal{U}$-modules  and the category of $\mathcal{U}(\mathfrak{g},e)$-modules.
\begin{align*}
 &\mathcal{C} \longrightarrow \mathcal{U}(\mathfrak{g},e)\mathrm{-mod};\; M \mapsto  M^{\mmm'},\\
 &\mathcal{U}(\mathfrak{g},e)\mathrm{-mod}\longrightarrow \mathcal{C}; \; N \mapsto \mathcal{S}(N): =Q_{\chi} \otimes N.
\end{align*}

\begin{theorem}\label{Skryabin's equivalence}
The functors given above are a pair of quasi-inverse equivalence.
\end{theorem}

This theorem was first stated in \cite{ZS1}. Following \cite{Lo1} we give an alternative proof to it by using the Poisson realization given above.

\begin{proof}

Let $M$  be a  Whittaker $\ggg$ module. Since $\mmm'$ acts locally nilpotent on $M$, we can view  $M$ as  a continuous $\mathcal{U}^{\wedge}_{\mmm'}$ module. Thus we can view $M$ as an $\mathbf{A}_V(\mathscr{W})$-module via the isomorphism $\Phi$. Now we have

\begin{equation}\label{eq 4.2}
M^{\mmm'}=M^{\mathfrak{I}}=M^{\Phi^{-1}(\mathfrak{I})}=M^{\mathfrak{I}_1}=M^{\bar{\mmm}}.
\end{equation}
Here $\mathfrak{I}$ (\textit{resp}. $\mathfrak{I}_{1}$) is the closure of $\mathcal{I}$ (\textit{resp}.  $\mathcal{I}_1$) in $\mathcal{U}^{\wedge}_{\mmm}$
(\textit{resp}. $\mathbf{A}_V(\mathscr{W})$).

For such an  $\mathbf{A}_V(\mathscr{W})$-module $M$, by Lemma \ref{non-super decom thm}  and its proof we have the following  $\mathbf{A}_V(\mathscr{W})$-module isomorphism
\begin{align} \label{eq 4.3}
 S(\bar{\mmm}^*) \otimes M^{\bar{\mmm}} &\longrightarrow  M \\
                   x \otimes m &\mapsto x\cdot m   \quad \text{for $x \in S(\bar{\mmm}^*)$ and $m \in M$.} \nonumber
\end{align}
It is clear that $Q_{\chi}=\mathcal{U}/\mathcal{U}\mmm'=\mathcal{U}^{\wedge}/\mathcal{U}\mathfrak{I}_1$. Hence by \eqref{eq 4.2} we have $\mathscr{W}=Q_{\chi}^{\bar{\mmm}}$. Thus Isomorphism \eqref{eq 4.3} implies  $Q_{\chi}=S(\bar{\mmm}^*)\otimes Q_{\chi}^{\bar{\mmm}}=S(\bar{\mmm}^*)\otimes \mathscr{W}$.

Finally we have
\begin{align*}
&\mathcal{S}(M^{\mmm'})=Q_{\chi}\otimes_{\mathscr{W}} M^{\bar{\mmm}}= S(\bar{\mmm}^*)\otimes \mathscr{W} \otimes_{\mathscr{W}}M^{\bar{\mmm}}=M\\
&(Q_{\chi}\otimes N)^{\mmm'}= (S(\bar{\mmm}^*)\otimes\mathscr{W}\otimes_{\mathscr{W}}\otimes N)^{\bar{\mmm}}=N.
\end{align*}
\end{proof}

For an associative algebra $\mathcal{A}$ and an $\mathcal{A}$-module $M$, denote by  $\mathrm{GK}_{\mathcal{A}}(M)$ the Gelfand-Kirillov dimension of
$M$.  The following theorem  gives the behavior of  $\mathrm{GK}$ dimension under the Skryabin's equivelance.

\begin{lemma}\label{LEMGK}
For any $\mathscr{W}$-module $N$, we have
$$\mathrm{GK}_{\mathcal{U}(\ggg_{\bar{0}})}(\mathcal{S}(N))=\mathrm{GK}_{\mathcal{U}}(\mathcal{S}(N))=\mathrm{GK}_{\mathscr{W}}(N)+\dim(\mmm'_{\bar{0}}).$$
\end{lemma}

\begin{proof}
Let $\mathcal{A},\mathcal{A}_{1}$  have the same meaning as in Subsection \S \ref{subsec con of w}.
Let $\mathcal{A}_{\hat{0}}$ be the subalgebra of $\mathcal{A}$  generated by $\ggg_{\bar{0}}$. Choose a finite-dimensional subspace $M_0$ generating $\mathcal{A}_{\hat{0}}^{\heartsuit}$ module $\mathcal{S}(M)$.
Then the lemma follows from the simple fact
$$\lim_{n\rightarrow \infty}  \frac{\ln \dim((F_n\mathcal{A})M_0)}{\ln n}=\lim_{n\rightarrow \infty}  \frac{\ln \dim(F_n(\mathcal{A}_{\hat{0}})M_0)}{\ln n},$$
along with the argument in the proof of \cite[Proposition 3.3.5]{Lo1}.
\end{proof}

\subsection{Maps $\bullet_{\dag}$ and $\bullet^{\dag}$}
For an associative algebra $\mathcal{A}$, we denote by $\mathfrak{id}(\mathcal{A})$ the set of its two-sided ideals. An ideal $\mathcal{I} \subset \mathcal{A}$  is said to be  primitive if it is the annihilator of an irreducible $\mathcal{A}$ module. The set of primitive ideals of $\mathcal{A}$ is denoted by $\Prim(\mathcal{A})$. Obviously $\Prim(\mathcal{A}) \subset \mathfrak{id}(\mathcal{A})$. Let $\Prim^{\fin}(\mathcal{A})$ be the set of primitive ideals of $\mathcal{A}$ with finite codimension. It is well known that there is a bijection between the set of finite-dimensional  irreducible $\mathcal{A}$ modules and $\Prim^{\fin}(\mathcal{A})$.

From now on, we  assume that $\mathfrak{g}$ is a basic Lie superalgebra of type \uppercase\expandafter{\romannumeral1}.
In this section we establish correspondences $\bullet^{\dag}$ and $\bullet_{\dag}$ between $\mathfrak{id}(\mathscr{W})$ and $\mathfrak{id}(\mathcal{U})$ following Losev in the non-super case. Then we study the finite-dimensional irreducible representations of $\mathscr{W}$ by them.

We simplify the notation $\mathbf{A}(V)_{\bar{\mmm}}^{\wedge} \otimes \mathscr{W}$ by $\mathbf{A}(\mathscr{W})$.
Following \cite{Lo1}, we construct a map $\mathfrak{id}(\mathscr{W}) \rightarrow  \mathfrak{id}(\mathcal{U}); \mathcal{I} \mapsto \mathcal{I}^{\dagger}$ as follows.
For $\mathcal{I} \in \mathfrak{id}(\mathscr{W})$ define an ideal of $\mathbf{A}(\mathscr{W})$ by
$$\mathbf{A}(\mathcal{I})^{\wedge}=\lim_{\leftarrow}(\mathbf{A}(\mathscr{W})(\mathcal{I})+\mathbf{A}(\mathscr{W})\bar{\mmm}^{k})/ \mathbf{A}(\mathscr{W})\bar{\mmm}^{k}.$$

Set $\mathcal{I}^{\dag}=\cu \cap \Phi(\mathbf{A}(\mathcal{I})^{\wedge})$.
Now we have the following  properties of  the map  $\bullet^{\dag}$.

\begin{theorem}\label{thm5.1}
\begin{itemize}
\item[(\rmnum{1})] $ (\mathcal{I}_1 \cap \mathcal{I}_2)^{\dag} = \mathcal{I}_1^{\dag} \cap \mathcal{I}_2^{\dag}$, for all $\mathcal{I}_1, \mathcal{I}_2 \in \mathfrak{id}(\mathscr{W})$.
\item[(\rmnum{2})]  For any $\mathscr{W}$-module $N$, $\Ann(N)^{\dag}=\Ann (\mathcal{S}(N))$.
\item[(\rmnum{3})] $\iota(\mathcal{I}^{\dag} \cap \mathcal{Z} (\ggg))=\mathcal{I} \cap \iota(\mathcal{Z} (\ggg))$. Here $\mathcal{Z} (\ggg)$ is the center of $\mathcal{U}$ and $\iota$ is the natural inclusion $\mathcal{Z} (\ggg) \hookrightarrow \mathscr{W}$.
\item[(\rmnum{4})] $\mathcal{I}^{\dag}$ is prime provided $\mathcal{I}$ is prime. The ideal $\mathcal{I}^{\dag}$ is primitive if $\mathcal{I}$ is prime with \\  $\mathrm{codim}_{\mathcal{Z}(\mathscr{W})}({\mathcal{Z}(\mathscr{W})} \cap \mathcal{I}) = 1$.
 \end{itemize}
\end{theorem}

\begin{proof}
We only prove the second statement of (\rmnum{4}). The proofs of the remaining statements are exactly the same as those of \cite[Theorem 1.2.2]{Lo1}.
For a prime ideal $\mathcal{I}\subset \mathscr{W}$ with $\codim_{\mathcal{Z}(\mathscr{W})}({\mathcal{Z}(\mathscr{W})} \cap \mathcal{I}) = 1$,
 statement (\rmnum{3}) implies $\codim_{\mathcal{Z}(\ggg)}(\mathcal{I}^{\dag} \cap \mathcal{Z} (\ggg))=1$. Since we are assuming $\ggg$ is of type \uppercase\expandafter{\romannumeral1} , we can prove by the same argument as in [\cite{Jan},Proposition 7.3] that any  prime ideal $\mathcal{J}\subset \mathcal{U}$ with $\codim_{\mathcal{Z}(\mathcal{\ggg})}(\mathcal{Z}(\mathcal{U})\cap \mathcal{J})=1$ is primitive.
\end{proof}

The following is another description of the map $\bullet^\dag$.
\begin{prop}\label{prop5.2}
We have $\mathbf{R}_{\hbar}(\mathcal{I}^{\dag})=\mathbf{R}_\hbar(\mathcal{U})\cap \Phi^{-1}_{\hbar}(S[[V,\hbar]]\otimes\bar{\mathcal{I}}_\hbar)$.
For a given $\mathcal{I} \in \mathfrak{id}(\mathscr{W})$, $\mathcal{I}^{\dag}$ is uniquely determined by the above property.
\end{prop}

We omit the proof  of the proposition, which may be given by the same argument as in the proof of \cite[Proposition 3.4.1]{Lo1}.

\subsection{Finite-dimensional representations of $\mathcal{U}(\ggg,e)$}\label{sec fd rep}

In this subsection we construct a series of  finite-dimensional $\mathcal{U}(\ggg,e)$-modules. Moreover, we prove that by such a  construction, all finite-dimensional irreducible representations of $\mathscr{W}$ can be exhausted (see Theorem \ref{main thm on rep}). Our another main tool, a map $\bullet_{\dag}:\mathfrak{id}(\mathcal{U}) \mapsto \mathfrak{id}(\mathscr{W})$ is defined as follows.
For an ideal $\mathcal{I} \in \mathfrak{id}$, denote by  $\bar{\mathcal{I}}_{\hbar}$  the closure of $\mathbf{R}_{\hbar}(\mathcal{I})$ in $S[\ggg]^{\wedge_{\chi}}_{\hbar}$. Define $\mathcal{I}_{\dag}$ to be the unique (by Proposition \ref{prop3.4}(3)) ideal in $\mathscr{W}$  such that

$$\mathbf{R}_{\hbar}(\mathcal{I}_{\dag})=\Phi_{\hbar}^{-1}(\bar{\mathcal{I}}_{\hbar}) \cap \mathbf{R}_{\hbar}(\mathscr{W})$$

\begin{prop}\label{propinclu}
For any $\mathcal{I} \in \mathfrak{id}(\mathscr{W})$  and $\mathcal {J} \in \mathfrak{id}(\mathcal{U})$ we have $\mathcal{I} \supset (\mathcal{I}^{\dag})_{\dag}$ and $\mathcal{J} \subset (\mathcal{J}_{\dag})^{\dag}$.
\end{prop}

\begin{proof}
We have

$$ \begin{array}{lll}
\mathbf{R}_\hbar((\mathcal{I}^{\dag})_\dag) &=\Phi_{\hbar}^{-1}(\overline{\mathbf{R}_\hbar(\mathcal{I}^{\dag})}) \cap \mathbf{R}_\hbar(\mathscr{W}) \\
                           &= \Phi_{\hbar}^{-1}(\overline{\Phi_{\hbar}(S[[V,\hbar]]\otimes \overline{\mathbf{R}_\hbar(\mathcal{I})}) \cap \mathbf{R}_\hbar(U)}) \cap \mathbf{R}_\hbar(\mathscr{W}) \\
                           & \subset    (S[[V,\hbar]]\otimes \overline{\mathbf{R}_\hbar(\mathcal{I})}) \cap \mathbf{R}_\hbar(\mathscr{W}) \\
                           &=\overline{\mathbf{R}_\hbar(\mathcal{I})} \cap \mathbf{R}_\hbar(\mathscr{W}) \\
                           &= \mathbf{R}_\hbar(\mathcal{I}).
 \end{array}$$

Here the first equation follows from  Proposition \ref{prop5.2} and the last one follows from Proposition \ref{prop3.4}(3). Hence we have $(\mathcal{I}^{\dag})_\dag \subset \mathcal{I}$.
$$ \begin{array}{lll}
\mathbf{R}_\hbar((\mathcal{J}_{\dag})^{\dag})&=    \Phi_\hbar^{-1} ( S[[V,\hbar]] \otimes \overline{\mathbf{R}_\hbar(\mathcal{J}_{\dag})}) \cap \mathbf{R}_\hbar(\mathcal{U}) \\
                                             &=    \Phi_\hbar^{-1} ( S[[V,\hbar]] \otimes (\Phi_\hbar (\overline{\mathbf{R}_\hbar(\mathcal{J}})\cap S[[\mathbf{S}, \hbar]])  \cap \mathbf{R}_\hbar(\mathcal{U}) \\
                                             &= \overline{\mathbf{R}_\hbar(\mathcal{J})} \cap \mathbf{R}_\hbar(\mathcal{U})  \\
                                             & \supset \mathbf{R}_\hbar(\mathcal{J}).
\end{array}$$

Here the third equation follows from the forthcoming fact \eqref{eq5.2}. Thus we have  $\mathcal{J} \subset (\mathcal{J}_{\dag})^{\dag}$.

\end{proof}

\begin{prop}\label{prop5.4}
The isomorphism $\Phi_\hbar$ induces an isomorphism between
$\gr(\mathcal{J}_{\dag})$ and $(\gr(\mathcal{J})+\mathbf{I}(\mathbf{S}))/\mathbf{I}(\mathbf{S})$, where $\mathbf{I}(\mathbf{S})$ is the ideal of $S[\ggg]$ generated by $V$.  Here $``\gr"$ means the associated grading with respect to the Kazadan filtration.
\end{prop}

\begin{proof}
The two sided  $\hbar$-saturated ideal $\Phi_{\hbar}(\bar{\mathcal{J}}_{\hbar})$ is stable to the adjoint action of $S[V,\hbar]^{\wedge_\chi}$.
From the proof of Theorem \ref{SQDWtm} we have
\begin{equation}\label{eq5.2}
\Phi_{\hbar}^{-1}(\bar{\mathcal{J}}_{\hbar})=S[V,\hbar]^{\wedge_\chi}\otimes (S[\tilde{\ggg}_{e},\hbar]^{\wedge_\chi}\cap \Phi_{\hbar}^{-1}(\bar{\mathcal{J}}_{\hbar}) ).
\end{equation}

Let $I^{\wedge}$, $\Phi_{\hbar}^{-1}((\bar{\mathcal{J}}_{\hbar}))^{\wedge}$ and $\mathbf{I}(\mathbf{S})^{\wedge}$ be the classical part of   $\mathbf{R}_{\hbar}(\mathcal{I}_{\dag})$,$\mathbf{R}_{\hbar}(\Phi_{\hbar}^{-1}(\bar{\mathcal{I}}_{\hbar}))$ and $\mathbf{I}(\mathbf{S})$ respectively. The equation \eqref{eq5.2} and the second statement of Lemma  \ref{decom of qun unive le} implies
$$ I^{\wedge} \simeq  (\Phi_{\hbar}^{-1}(\bar{\mathcal{J}}_{\hbar}))^{\wedge}+\mathbf{I}(\mathbf{S})^{\wedge})/\mathbf{I}(\mathbf{S})^{\wedge}.$$

Since $\gr(\mathcal{I}_\dag)$ is dense in $I^{\wedge}$, by the second statement of Lemma  \ref{decom of qun unive le} and  the third statement of Proposition \ref{prop3.4} we have
$$\gr(\mathcal{I}_{\dag}) \simeq (\gr(\mathcal{J})+\mathbf{I}(\mathbf{S}) )/\mathbf{I}(\mathbf{S}).$$
\end{proof}

Let $\mathfrak{v}=\mathfrak{v}_{\bar{0}}+\mathfrak{v}_{\bar{1}}$ be a superspace and $\mathcal{A}$ be a filtered associative algebra with $\gr(\mathcal{A})=S[\mathfrak{v}]$. For a two sided ideal $\mathcal{J} \subset \mathcal{A}$, we denote by $\mathbf{V}(\mathcal{J})$ the maximal
spectrum of $S[\mathfrak{v}_{\bar{0}}]/\gr(\mathcal{J} \cap S[\mathfrak{v}_{\bar{0}}])$. For an $\mathcal{A}$ module $M$ with annihilator $\mathcal{J}$, the associated variety $\mathbf{AV}(M)$ of $M$ is defined to be $\mathbf{V}(\mathcal{J})$. The set of the primitive ideals $\mathcal{J} \subset \mathcal{U}$ with $\mathbf{V}(\mathcal{J})=\overline{G_{\bar{0}}\cdot \chi}$ is denoted by $\Prim_{\mathbb{O}}(\mathcal{U})$.
Let $\mathbf{S}_{\bar{0}}$ be the Slodowy slice of $\mathcal{O}$ in $\mathfrak{g}_{\bar{0}}$.
\begin{theorem} \label{thm5.3}
For any $\mathcal{J} \in \Prim_{\mathbb{O}}(\mathcal{U})$, the composition factors of $\mathscr{W}/\mathcal{J}_{\dag}$
are finite-dimensional irreducible $\mathcal{U}(\ggg,e)$-modules. Any finite-dimensional irreducible $\mathcal{U}(\ggg,e)$-module is can be obtained in this way.
\end{theorem}

\begin{proof}
 By Proposition \ref{prop5.4} we have $\gr(\mathcal{J}_{\dag})=(\gr(\mathcal{J})+\mathbf{I}(\mathbf{S}))/\mathbf{I}(\mathbf{S})$.
Thus we have
\begin{equation}\label{eq5.1}
\gr(\mathcal{J}_{\dag})\cap \mathbb{C}[\mathbf{S}_{\bar{0}}]=(\gr(\mathcal{J})\cap S[\ggg_{\bar{0}}] +\mathbf{I}(\mathbf{S})\cap S[\ggg_{\bar{0}}])/\mathbf{I}(\mathbf{S})\cap S[\ggg_{\bar{0}}].
\end{equation}
This implies
$$\mathbf{V}(\gr(\mathcal{J}_{\dag}))=\mathbf{S}_{\bar{0}}\cap \mathcal{O}=\chi.$$
Hence $\mathcal{J}_\dag$ has finite codimension in $\mathscr{W}$, i.e. $\mathscr{W}/\mathcal{J}_{\dag}$ is a finite-dimensional $\mathscr{W}$-module.
So we prove the first statement.

Let $M$ be an irreducible $\mathscr{W}$-module and $\mathcal{I} \in \Prim^{\fin}(\mathscr{W})$ be
its annihilator. By Proposition \ref{propinclu} we have the following canonical surjective homomorphisms $\mathscr{W}/((\mathcal{I}^\dag)_\dag) \twoheadrightarrow \mathscr{W}/(\mathcal{I}) \twoheadrightarrow M$.  Thus the theorem follows.
\end{proof}

\begin{theorem}\label{main thm on rep}
For any $\mathcal{I} \in \Prim^{\fin}(\mathscr{W})$, we have $\mathcal{I}^{\dag}\in \Prim_{\mathbb{O}}\mathcal{U}$;
for any $\mathcal{J} \in\Prim_{\mathbb{O}}\mathcal{U}$, there are finitely many  $ \mathcal{I} \in \Prim^{\fin}(\mathscr{W})$ with $\mathcal{J}=\mathcal{I}^{\dag}$.
\end{theorem}

\begin{proof}
Let $\mathcal{I} \in \Prim^{\fin}\mathscr{W}$ be the annihilator of a given irreducible $\mathscr{W}$ module $M$.
By Theorem \ref{thm5.1} (\rmnum{2}), we have $\mathcal{I}^{\dag}=\Ann(\mathcal{S}(M))$. Since $\mathcal{S}(M)$ is a Whittaker module for the pair $(\ggg_{\bar{0}},e)$, we have $\overline{G_{\bar{0}}\cdot\chi} \subset\mathrm{V}(\mathcal{I}^{\dag})$ \cite[Theorem 3.1]{Pr2}.

Since the associated variety of an irreducible $\mathcal{U}(\ggg_{\bar{0}})$ module is the closure of a nilpotent orbit, there exists an irreducible sub-quotient $M^{\dag}$ of $\mathcal{S}(M)$ with $\mathbf{AV}(M^{\dag})=\mathbf{V}(\mathcal{I}^{\dag})$. For the annihilator $\mathcal{I}_0^{\dag}\subset \mathcal{U}(\ggg_{\bar{0}})$ of $M^{\dag}$, we have $\dim(V(\mathcal{I}_0^{\dag}))\leq 2\mathrm{GK}_{\mathcal{U}\ggg_{\bar{0}}}(M^{\dag})$ by \cite[\S10.7]{Jan}.  By \cite[Lemma 7.3.3(a)]{Mus} we have $\mathrm{GK}_{\mathcal{U}(\ggg_{\bar{0}})}(M^{\dag})\leq \mathrm{GK}_{\mathcal{U}(\ggg_{\bar{0}})}(\mathcal{S}(M))$. By Lemma \ref{LEMGK} we have $\mathrm{GK}_{\mathcal{U}\ggg_{\bar{0}}}(\mathcal{S}(M))=\dim(\mmm_{\bar{0}})$.  The later equals $\dim(G_{\bar{0}}\cdot \chi)/2$. Thus finally
we have $\dim(V(\mathcal{I}^{\dag})) \leq \dim(G_{\bar{0}}\cdot \chi)$ and hence the first statement follows.

Suppose $\mathcal{J}=\mathcal{I}^\dag$ for some $\mathcal{I} \in \Prim^{\fin}\mathscr{W}$. By Proposition \ref{propinclu} we have $\mathcal{I}\supset \mathcal{J}_\dag$.   By the proof of Theorem \ref{thm5.3}, $\mathcal{J}_\dag$ has finite codimension in $\mathscr{W}$.  Thus any prime ideal containing $\mathcal{J}_\dag$ is minimal. So by \cite[Proposition 3.1.10]{Dix} the number of primitive ideals $ \mathcal{I} \in \Prim^{\fin}\mathscr{W}$ with $\mathcal{J}=\mathcal{I}^\dag$ is finite.
 \end{proof}

 Finally we summarize our results on $\mathrm{Irr}^{\text{fin}}(\mathscr{W})$ in this section.
 For a Lie superalgebra $\ggg=\ggg_{\bar{0}} \oplus \ggg_{\bar{1}}$  of type \uppercase\expandafter{\romannumeral1}, Theorem \ref{main thm on rep} provides us a map
$$\Prim^{\fin}(\mathscr{W}) \rightarrow \Prim_{\mathbb{O}}\mathcal{U}; \quad  \mathcal{I} \mapsto \mathcal{I}^\dag $$
with finite fiber. For the universal enveloping algebra $\mathcal{U}$ of $\ggg$,  Leizter established  a bijection
$$ \nu : \Prim(\mathcal{U})\longrightarrow \Prim(\mathcal{U}(\ggg_{\bar{0}})).$$
We refer  the construction of $\nu$ to  \cite[\S3.3]{Le}  or \cite[\S15.2.2]{Mus}.  It can be easily deduced from the construction that $\nu(\mathcal{I})$ is supported on $G_{\bar{0}}\cdot \chi$ if and only if so is $\mathcal{I}$. Thus we make a step forward a description of the set $\mathrm{Irr}^{\text{fin}}(\mathscr{W})$.


\begin{thebibliography} {ABCD}
\bibitem[BBG]{BBG}
J. Brown, J. Brundan and S. Goodwin, \textit{Principal $W$-algebras for
$\GL(m|n)$}, Algebra Number Theory, 7(2013), 1849-1882.

  \bibitem[BG]{BG} J. Brundan, S. M. Goodwin, {\em  Whittaker coinvariants for $GL(m|n)$}, Adv. Math.  347 (2019), 273-339.

\bibitem[CG]{CG} N. Chriss, V. Ginzburg, {\em Representation theory and complex geometry}, Birkhauser, 1997.
 \bibitem[CM]{CM} K. Coulembier, I. Musson, {\em  The primitive spectrum for $\mathfrak{gl}(m|n)$}. Tohoku Math. J.(2)70 (2018), no. 2, 225-266.
\bibitem[Dix]{Dix}J. Dixmier, {\em Enveloping algebras}. North-Holland Mathematical  Library, Vol. 14.

\bibitem[Hoy]{Hoy} C. Hoyt, {\em Good gradings of basic Lie superalgebras}, Israel J. Math., 192 (2012), 251-280.
\bibitem[Jan]{Jan} J. C. Jantzen, {\em  Einhüllende Algebren halbeinfacher Lie-Algebren}. Ergebnisse der Math., Vol. 3, Springer, New York, Tokyo, 1983.
\bibitem[Kal]{Kal} D. Kaledin, {\em Symplectic singularities from the Poisson point of view}, Journal Für Die Reine Und Angewandte Mathematik, 600 (2006), 135-156.
\bibitem[Kos]{Kos} B. Kostant, {\em Graded manifolds, graded Lie groups and prequantization},  LMN 570(1975), Springer, 109-177.


\bibitem[Le]{Le}   E. Lezter, {\em A bijection of primitive spectra for classical Lie superalgebras of type I}, J. London Math. Soc.,  53 (1996),  39-49.
 \bibitem[Lo1]{Los} I. Losev, {\em Lectures on finite $W$-algebras at East China Normal University}, Shanghai, July 2016.
\bibitem[Lo2]{Lo1} I. Losev, {\em  Quantized symplectic actions and W-algebras}, J. Amer. Math. Soc., 23 (2010), 35-59.
    \bibitem[Lo3]{Lo2} I. Losev, {\em  Primitive ideals for W-algebras in type A}, Journal of Algebra, 359 (2012), 80-88.                                         \bibitem[Lo4]{Lo3} I. Losev, {\em Finite-dimensional representations of W-algebras}, Duke Math. J., 159 (2011),  99-143.
   \bibitem[Lo5]{Lo4} I. Losev, {\em  Dimensions of irreducible modules over W-algebras and Goldie ranks}, Inventiones Math., 200 (2015), 849-923.
\bibitem[Mus]{Mus} I. Musson,  {\em Lie Superalgebras and Enveloping Algebras}, Graduate Studies in Mathmaticas GSM 131, American Mathematical Society 2012.
\bibitem[Pr1]{Pr1} A. Premet,  {\em  Special transverse slices and their enveloping algebras}, Advances in Mathematics, 170 (2002), 1-55.
\bibitem[Pr2]{Pr2} A. Premet,  {\em  Enveloping algebras of Slodowy slices and the Joseph ideal},  J. Eur. Math. Soc., 9 (2007), 487-543.
\bibitem[PS1]{PS1} E. Poletaeva,  V. Serganova, {\em  On Kostant’s theorem for the Lie superalgebra Q(n)}, Adv. Math., 300 (2016), 320-359.
 \bibitem[PS2]{PS2} E. Poletaeva,  V. Serganova, {\em  On the finite $W$-algebra for the Lie superalgebra Q(N) in the non-regular case}, J. Math. Physics, 58 (2017), 111701,22 pp.

     \bibitem[PS3]{PS3} E. Poletaeva,  V. Serganova, {\em Representations of principal W -algebra for the superalgebra Q(n)}, arXiv: 1903.05272v1[math.RT]
   \bibitem[SXZ]{SXZ}  B. Shu, H. Xiao, Y. Zeng, {\em  On classification of finite-dimensional irreducible modules for  finite $W$-superalgebra}, in preparation.
\bibitem[WZ]{WZ}  W. Wang, L. Zhao, {\em  Representations of Lie superalgebras in prime characteristic I}, Proc. London Math. Soc., 99 (2009), 145-167.
\bibitem[Wei]{Wei} A. Weinstein, {\em The local structure of Poisson manifolds } J. Differential Geom., 18 (1983),  523-557.
\bibitem[ZS1]{ZS1} Y. Zeng, B. Shu, {\em Finite $W$-superalgebras for basic Lie superalgebras}, J. Algebra, 438 (2015), 188-234.
\bibitem[ZS2]{ZS2} Y. Zeng, B. Shu, {\em Finite  $W$-superalgebras and the dimensional lower bounds for the representations of basic Lie superalgebras}, Publ. RIMS Kyoto Univ., 53 (2017), 1-63.
\bibitem[ZS3]{ZS3} Y. Zeng, B. Shu,  {\em Minimal $W$-superalgebras and the modular representations of basic Lie superalgebras}, Publ. RIMS Kyoto Univ., 55 (2019), 123-188.
\end{thebibliography}
\end{document}